\documentclass[a4paper,11pt]{amsart} 
\usepackage[utf8]{inputenc}
\usepackage[T1]{fontenc}
\usepackage[english]{babel}
\usepackage{mathtools,amssymb,amsthm}
\usepackage{enumitem}
\usepackage{esint}
\usepackage{lmodern}
\usepackage{geometry} 
\usepackage{amsbsy} 
\usepackage{mathrsfs}

\usepackage{color}

\newcommand \dx{\mathrm{d}x}

\newcommand \ds{\mathrm{d}{\mathscr H}^{n-1}}

\newcommand \Lb{\Lambda}
\newcommand \Hs{\mathscr{H}^{n-1}}
\newcommand \om{\omega}
\newcommand \ra{\rightarrow}
\newcommand \sq{\subseteq}
\newcommand \Om{\Omega}

\newcommand \eps{\epsilon}
\newcommand \Per{\text{Per}}
\newcommand \goto{\underset{\eps\to 0}{\longrightarrow}}
\newcommand \R{\mathbb{R}}
\newcommand \E{\mathscr{E}}
\newcommand \dv{\mathrm{div}}
\newcommand \loc{\mathrm{loc}}
\newcommand \Rn{\mathbb{R}^n}

\newcommand \RP{\mathbb{RP}^{n-1}}

\newtheorem{theorem}{Theorem} 
\newtheorem{corollary}[theorem]{Corollary} 
\newtheorem{lemma} [theorem]{Lemma} 
\newtheorem{remark} [theorem]{Remark} 
\newtheorem{proposition} [theorem]{Proposition} 
 
\theoremstyle{definition}

\begin{document}

\title[\ ]{Shape optimization of a thermal insulation problem}

\author[D. Bucur]
{Dorin Bucur}
\address[Dorin Bucur]{Univ. Savoie Mont Blanc, CNRS, LAMA \\
73000 Chamb\'ery, France
}
\email[D. Bucur]{  dorin.bucur@univ-savoie.fr}

\author[M. Nahon]{Mickaël Nahon}
\address[Mickaël Nahon]{Univ. Savoie Mont Blanc, CNRS, LAMA \\ 73000 Chamb\'ery, France}
\email{  mickael.nahon@univ-smb.fr}

\author{Carlo Nitsch}
\address[Carlo Nitsch]{University of Napoli Federico II \& Scuola Superiore Meridionale \\
80126 Naples, Italy
}
\email{  c.nitsch@unina.it}

\author{Cristina Trombetti}
\address[Cristina Trombetti]{University of Napoli Federico II \\
80126 Naples, Italy
}
\email{  cristina@unina.it}

\maketitle

\begin{abstract} 
We study a shape optimization problem involving a solid $K\subset\R^n$ that is maintained at constant temperature and is enveloped by a layer of insulating material $\Om$ which obeys a generalized boundary heat transfer law. 
We minimize the energy of such configurations among all $(K,\Om)$ with prescribed measure for $K$ and $\Om$, and no topological or geometrical constraints. In the convection case (corresponding to Robin boundary conditions on $\partial\Om$) we obtain a full description of minimizers, while for general heat transfer conditions, we prove  the existence and regularity of solutions and give a partial description of minimizers.

\end{abstract}

\section{Introduction} 
 Given a measurable set $K \sq \R^n$ along with a Lipschitz open set $\Omega\supset K$, we consider the energy functional
\begin{equation}\label{eq_optintro}
E_{\Theta}(K,\Om):=\min\left\{ \int_{\Om}|\nabla v|^2\dx+\int_{\partial\Om}\Theta(v)\ds,\ v \in H^1(\Om,[0,1]): v = 1 \mbox{ a.e. on } K\right\}
\end{equation}
where $\Theta:[0,1]\to\R_+$ is a rather general nondecreasing function that vanishes at $0$.  
If  all data are smooth, meaning $\Om$ is a smooth open set, $K$ is a smooth compact set and $\Theta\in\mathcal{C}^1$, a minimizer $u$ 

 satisfies 
\begin{equation}\label{smoothcase}
\begin{cases}
\Delta u=0 & \Om\setminus K,\\
-\frac{\partial u}{\partial \nu}=\frac{1}{2}\Theta'(u) &\partial\Om\cap\{u>0\}\setminus\partial K,\\
-\frac{\partial u}{\partial \nu}\leq\frac{1}{2}\Theta'(u) &\partial\Om\cap\{u=0\},\\
u\equiv 1 & K.
\end{cases}
\end{equation}
A classical prototype, $\Theta (u)= \beta u^2$  for some $\beta>0$, corresponds to the so called Robin boundary conditions $\frac{\partial u}{\partial \nu} +\beta u=0$ for the harmonic function $u$ minimizing the energy,  $\nu$ being the outward normal at the boundary.
The physical motivation which leads us to study the functional $E_{\Theta}(K,\Om)$ and related shape optimization problems can be found in optimal design in thermal insulation.\\
Consider a body $K$ of given constant fixed temperature $T_K$, surrounded by an insulator $A$. The temperature distribution $T$ inside the insulator satisfies the classical heat equation (Laplace equation) with the condition that $T$ is continuous across the boundary which separates $A$ and $K$. If we now assume that the body $\Omega=K\cup A$ is immersed into an environment of fixed temperature $T_e$, the heat exchange rate across the surface of $\Omega$ (between the body and the environment) has to be modeled according to the physical process governing the mechanism. The most common assumption is to assume that the temperature at the boundary of $\Omega$ is kept constant ($T_e$), leading to the so-called Dirichlet boundary conditions, which corresponds to conduction heat transfer. But if we assume that the environment is a fluid (gas or liquid) then, convection heat transfer, radiation heat transfer or even more general laws, have to be taken into account. If for instance, convection heat transfer is the leading mechanism, then the rate of heat flux per unit of surface, across the solid fluid interface, is proportional to $T|_{\partial \Omega}-T_e$ (also known as Newton's law of cooling), $T|_{\partial \Omega}$ being the temperature $T$ at the solid surface. While for radiation heat transfer (according to Stefan-Boltzmann law), the heat flux per unit of surface, across the solid fluid interface, is proportional to $T_{\partial \Omega}^4-T_e^4$.

On the other hand, the heat flux across a solid surface is proportional (Fourier's law) to the normal derivative of the temperature $-\frac{\partial T}{\partial \nu}$.

Assume that $T_K>T_e$ and denote by $u=\frac{T-T_e}{T_K-T_e}$.
One can easily check that, up to a constant of proportionality, convection corresponds to $$\Theta(u)=u^2.$$ 

Radiation, on the other hand, is modeled up to a multiplicative constant by 
\begin{equation}\label{bnnt100}
\Theta(u)=\frac{u^5}{5}+\gamma u^4+2\gamma^2 u^3+2\gamma^3 u^2,
\end{equation}
where $\gamma=\frac{T_e}{T_K-T_e}$.

Both mechanisms can be taken into account simultaneously, upon considering a linear combination of the previous functions. 
When considering a quadratic $\Theta$ (pure convection heat transfer at the boundary), $E_{\Theta}$ is proportional to the heat loss rate. In general, regardless of the choice of $\Theta$, the energy functional $E_{\Theta}$ can be considered a measure of the goodness of the thermal insulation; the less the energy the better the insulation.

In this article we are interested in the shape optimization problem of both $K$ and $\Om$ of prescribed volume which lead  to a minimal energy configuration. For a given set $K$, when only the geometry of $\Om$ is unknown,  the problem has already been considered in the literature in \cite{BL14}, \cite{CK16}, \cite{Kr19}, with a different purpose, namely to obtain qualitative information on the free boundary of $\Om$. Its analysis relies on the fine study of optimal configurations in the framework of free discontinuity problems in $SBV$.  

We search to optimize both $K$ and $\Om$ and the purpose is to understand whether or not an optimal configuration is given by two concentric balls (as common sense suggests). From this perspective, our problem is more of isoperimetric type. As we shall prove, depending on the dissipation law $\Theta$, we can give either full description of the optimal  sets as two concentric balls (for instance in the convection case) or some partial answer leading to the same geometric configuration for more general dissipation laws. A striking phenomenon, which has been observed in some specific geometric configurations in \cite{DNT20},  and that we have to handle is the following: if only a small amount of insulator is available, then in some cases it is better to not use it ($\Omega=K$).

Let us denote by $\om_n$ the volume of the unit ball and by $M:= R^n \om_n$ the volume of the ball of radius $R \ge 1$. Without restricting generality, we shall fix the measure of $K$ to be $\om_n$ and leave the other parameters free (the general case is obtained by rescaling). Given $M  \ge \om_n$ and $\Lb >0$, we are interested in the following minimization problem: find an open set $\Om$ and a relatively closed set $K$, solutions of
\begin{equation}\label{bnnt01}
\inf_{K\subset \Om,\ |K|= \om_n,\ |\Om|\leq M}E_{\Theta}(K,\Om),
\end{equation}
or in its penalized version
\begin{equation}\label{bnnt02}
\inf_{K\subset \Om,\ |K|= \om_n}E_{\Theta}(K,\Om)+ \Lb |\Om|.
\end{equation}

 Throughout the paper we assume that $\Theta:[0,1] \ra \R $ is a  lower semicontinuous, nondecreasing function such that $\Theta(0)=0$.    Here are our main results.

\begin{theorem}[The convection case]\label{bnnt03}
For $\Theta(u)=\beta u^2$  and $M=R^n\om_n$, the solutions of problems \eqref{bnnt01} consist of two concentric balls. The radius of the outer ball equals either $R$ or $1$, according to $\min\{E_\Theta (B_1,B_R),E_\Theta (B_1,B_1) \}$ and the associated state function $u$ is radial. 

\end{theorem}

\begin{theorem} \label{bnnt04} 
For any  $\Lambda>0$ and any admissible $\Theta$, the solution of problem \eqref{bnnt02} consists of two concentric balls
and the associated state function $u$ is radial.

\end{theorem}
We introduce the following hypothesis 
\begin{equation}\label{HypTheta}
 \inf_{0<s<1}\frac{\Theta(s/3)}{\Theta(s)}>0.
\end{equation}

\begin{theorem}[The general case] \label{bnnt03.2} 
 There exists some dimensional constant $c_n$, such that if $\Theta$ satisfies  hypothesis \eqref{HypTheta} and
\begin{equation}\label{HypT}
M<\om_n+ c_n\left(\inf_{0<s<1}\frac{\Theta(s/3)}{\Theta(s)}\right)^{2n}\int_{0}^{1}\frac{t^{2n-1}\mathrm{d}t}{\Theta(t)^{n}},
\end{equation}
then problem \eqref{bnnt01} has a solution $(K,\Om)$. If $|\Om|< M$ then $(K, \Om)$ are two concentric balls.
Otherwise, $\Om$ is an open set with rectifiable topological boundary such that $\Hs(\partial\Om)<\infty$ and  $K$ is relatively closed in $\Om$ with locally finite perimeter in $\Om$. The temperature $u\in H^1(\Om)$ is   $\mathcal{C}^{0,\frac{2}{n+2}}_\loc(\Om)$. In two dimensions, $\partial K\cap\Om$ is analytic.  

If moreover $\Theta$ is $\mathcal{C}^1$ in a neighbourhood of $1$ and
\[\frac{\Theta'(1)^2}{\Theta(1)}< 4( n-1),\]
then there exists $M >\om_n$ depending on $n$ and $\Theta$ such that the solution of problem \eqref{bnnt01} is $K= \Om=B_1$.

\end{theorem}

\begin{remark}\rm Notice that any homogenous functions $\Theta(u)= u^\alpha$, $\alpha >0$ satisfies the assumption \eqref{HypTheta}. Moreover,   if $\Theta (u) = O(u^2)$ then in \eqref{HypT} any value $M \in [\om_n, +\infty)$ is accepted. In the convection case $\Theta(u) =\beta u^2$ the full picture of the solutions is understood. Other interesting choices of $\Theta$ may be given by \eqref{bnnt100}, corresponding to thermal radiation, or $\Theta(u)=cu$ corresponding to a constant heat flux. Moreover, discontinuous functions $\Theta$, with $\Theta(0_+)>0$, are admissible. This means that one can consider functions like $\Theta(u)= c_1 1_{u>0}+c_2u^\alpha$ on $[0, 1]$ which models a cost of the highly insulating material  on $\partial (A\cup K)$, cost which is proportional to its surface measure. \newline

 The function $\Theta$ may be extended to $\R$ as a nondecreasing and nonnegative function, and the constraint  $1_K\leq v\leq 1$ a.e.   in \eqref{eq_optintro} may be relaxed into $v_{|K}\geq 1$ a.e. Whatever the extension of $\Theta$ is outside $[0,1]$,  the new problem is equivalent to the original one, by truncation below 0 and above 1.
\end{remark}
Clearly, we know many more things on the penalized problem \eqref{bnnt02} than on the constrained problem \eqref{bnnt01}, in particular that the solutions are always concentric balls. 
This stills leaves open the following question: 
\begin{center}
\textit{Under reasonable hypotheses on $\Theta$, is it true that the solution of \eqref{bnnt01} always consists  on two concentric balls?} 
\end{center}

The organisation of the paper is as follows.
\begin{itemize}[label=\textbullet]\setlength\itemsep{1em}
\item In  Section \ref{bnnt06} we discuss the convection case  and   we prove Theorem \ref{bnnt03}.  The proof  follows the strategy of  Bossel and Daners for the Faber-Krahn inequality for the Robin Laplacian and is based on the construction of a so called $H$-function. Up to knowledge of the authors, this is the only case (aside from Faber-Krahn) where this strategy works. However, we point out that this strategy is fully working only in dimension $2$, while for $n \ge 3$ it works only for $\beta > n-2$. This section is mostly independent from the rest of the paper, except for the case $n\ge 3$, $\beta \in (0, n-2]$ which is a consequence of the analysis by free discontinuity techniques, for which we refer to Corollary \ref{CorPenalized} and Remark \ref{bnnt20}.

\item Section \ref{bnnt07} is devoted to the analysis of the existence of relaxed solutions for the constrained problem \eqref{bnnt01} in the context of a general dissipation functions $\Theta$, and to the regularity of the free boundaries. 
 These results are rather technical and they prepare the   proofs of Theorems \ref{bnnt04} and  \ref{bnnt03.2}. The key idea is that once we know that problem \eqref{bnnt01} has a sufficiently smooth solution, we can extract qualitative information out of  its optimality. We work in the framework of free discontinuity problems, based on the analysis of special functions with bounded variation.
 
\item In Section  \ref{bnnt08} we prove Theorem \ref{bnnt04} and in Section  \ref{bnnt09} we prove Theorem \ref{bnnt03.2}.

\end{itemize}

\section{The convection case: proof of Theorem \ref{bnnt03}}\label{bnnt06}

In this section we consider  $\Theta(s)=\beta s^2$. In this case, the energy $E_\Theta$ is simply denoted by $E_\beta$ and takes the form
\begin{equation}\label{eq_Ebeta}
E_{\beta}(K,\Om):=\inf_{v\in H^1(\Om),v\geq 1_K}\int_{\Om}|\nabla v|^2\dx+\beta\int_{\partial\Om}v^2\ds.
\end{equation}

For fixed $K$, minimizers of the functional $\Om \to E_{\beta}(K,\Om)+\Lb |\Om|$ have been studied in \cite{BL14}, \cite{CK16} and \cite{Kr19}, in particular the existence of an optimal set $\Om$ and its regularity.

The claim of  Theorem \ref{bnnt03} is that the solution of the constrained problem \eqref{bnnt01} consists of two centered balls for every $M> \om_n$, and the size of the balls is then given by the study of the function
\[R\mapsto E_{\beta}(B_1,B_R).\]
Let us first describe this function in detail. We let 
\[\Phi_n(\rho)=\begin{cases}\log(\rho)&\text{ if }n=2\\ -\frac{1}{(n-2)\rho^{n-2}}&\text{ if }n\geq 3,\end{cases}\]
where the sign convention is taken such that $\Phi_n$ is always increasing. Relying on the expression of radial harmonic functions ($x\mapsto a+b\Phi_n(|x|)$) and the boundary conditions we obtain that the temperature associated to $(B_1,B_R)$ is given by
\[u(x)=1-\frac{\beta(\Phi_n(|x|)-\Phi_n(1))_+}{\Phi_n'(R)+\beta(\Phi_n(R)-\Phi_n(1))}\]
and
\[E_{\beta}(B_1,B_R)=\frac{\beta\Per(B_1)\Phi_n'(1)}{\Phi_n'(R)+\beta(\Phi_n(R)-\Phi_n(1))}.\]
Notice in particular that
\[\frac{d}{dR}E_{\beta}(B_1,B_R)\leq 0\text{ iff }\frac{d}{dR}(\Phi_n'(R)+\beta \Phi_n(R))\geq 0\text{ iff }R\geq \frac{n-1}{\beta}.\]
Moreover the extremal values are
\[E_{\beta}(B_1,B_1)=\beta n \om_n,\ \lim_{R\to\infty}E_{\beta}(B_1,B_R)=(n-2)n \om_n.\]
In two dimension there are two cases
\begin{itemize}[label=\textbullet]\setlength\itemsep{1em}
\item if $\beta\ge 1$ then $R\in [1,+\infty[\mapsto E_{\beta}(B_1,B_R)$ is decreasing.
\item if $ \beta<1$ then $R\in [1,+\infty[\mapsto E_{\beta}(B_1,B_R)$ increases on $[1,\beta^{-1}]$ and  decreases on $[\beta^{-1},+\infty[$, with the existence of a unique $R_{\beta}>\beta^{-1}$ such that $E_{\beta}(B_1,B_{R_{\beta}})=E_{\beta}(B_1,B_1)$.
\end{itemize}
In higher dimension there are three cases
\begin{itemize}[label=\textbullet]\setlength\itemsep{1em}
\item if $\beta \geq n-1 $ then $R\in [1,+\infty[\mapsto E_{\beta}(B_1,B_R)$ is decreasing.
\item if $ n-1 >\beta >n-2$ then $R\in [1,+\infty[\mapsto E_{\beta}(B_1,B_R)$ increases on $[1,\frac{n-1}{\beta}]$, decreases on $[\frac{n-1}{\beta},+\infty)$, with the existence of a unique $R_{\beta}>\frac{n-1}{\beta}$   such that $E_{\beta}(B_1,B_{R_{\beta}})=E_\beta(B_1,B_1)$.
\item if $\beta \leq n-2 $ then $R\in [1,+\infty[\mapsto E_{\beta}(B_1,B_R)$ reaches its minimum at $R=1$.
\end{itemize}

We postpone the analysis of the last case, and start by proving the result for the first two cases. For this we will need the following.
\begin{lemma}\label{bnnt13}
Let $R>1$, $\beta>0$, and let $u^*$ be the (unique) minimizer in \eqref{eq_Ebeta}  on $(B_1,B_R)$. Then $\frac{|\nabla u^*|}{u^*}\leq \beta$ on $B_R\setminus B_1$ if and only if
\[\forall \rho\in [1,R],\ E_{\beta}(B_1,B_\rho)\geq E_{\beta}(B_1,B_R).\]
\end{lemma}
\begin{proof}
We remind the expression
\[u^*(x)=1-\frac{\beta(\Phi_n(|x|)-\Phi_n(1))}{\Phi_n'(R)+\beta(\Phi_n(R)-\Phi_n(1))},\]
so by straightforward computations
\[\frac{|\nabla u^*|}{u^*}\leq\beta\text{ iff }\Phi_n'(\rho)+\beta\Phi_n(\rho)\leq \Phi_n'(R)+\beta\Phi_n(R) \text{ iff }E_{\beta}(B_1,B_\rho)\geq E_{\beta}(B_1,B_R).\]
\end{proof}

We may now prove the result. The proof relies on the study of the so called $H$ function introduced by Bossel \cite{B88}, see also \cite{D06} and \cite{BG10}, for the study of the Faber-Krahn inequality involving the first eigenvalue of the Laplace operator with Robin boundary conditions.
\smallskip

\begin{proof} [Proof of Theorem \ref{bnnt03}]

Let $(K,\Om)$ be smooth sets such that $ K \sq \Om$, $|K|=\om_n$, $|\Om|\le M$ (and $R$ the radius defined by $M=|B_R|$).  Let $u$ be the minimizer of $E_\beta (K, \Om)$ and denote $\Om_t=\{u>t\}$. We decompose $\partial \Om_t$ into two disjoint (up to a $\Hs$-negligible set) sets; $\partial^i\Om_t=\{u=t\}\cap\Om$, and $\partial^e\Om_t=\partial\Om_t\cap\partial \Om$. Then  for a.e. $t \in [0, 1]$ 
\begin{align*}
0=\int_{\{t<u<1\}}\frac{\Delta u}{u}\dx&=\int_{\{t<u<1\}}\left(\nabla\cdot\left(\frac{\nabla u}{u}\right)-\nabla u\cdot\nabla\frac{1}{u}\right)\dx\\
&=\int_{\partial\{t<u<1\}}\nu_{\Om_t}\cdot\frac{\nabla u}{u}\ds+\int_{\Om_t}\frac{|\nabla u|^2}{u^2}\dx\\
&=\int_{\partial K\cap\Om}|\nabla u|\ds-\int_{\partial^i\Om_t}\frac{|\nabla u|}{u}\ds\\
&\ -\beta\Hs(\partial\{t<u<1\}\cap\partial\Om)+\int_{\Om_t}\frac{|\nabla u|^2}{u^2}\dx.\\
\end{align*}
Since
\begin{align*}
E_\beta(K,\Om)&=\int_{\Om\setminus K}\nabla\cdot (u\nabla u)\dx+\beta\int_{\partial\Om}u^2\ds\\
&=\int_{\partial K\cap\Om}|\nabla u|\ds+\int_{\partial\Om\setminus\partial K}\left(\beta u^2+u\partial_\nu u\right)\ds+\int_{\partial K\cap\partial\Om}\beta u^2\\
&=\int_{\partial K\cap\Om}|\nabla u|\ds+\beta\Hs(\partial K\cap\partial\Om)
\end{align*}
then injecting this in the previous equation we obtain
\[E_{\beta}(K,\Om)=\beta\Hs(\partial^e\Om_t)+\int_{\partial^i\Om_t}\frac{|\nabla u|}{u}\ds-\int_{\Om_t}\frac{|\nabla u|^2}{u^2}\dx.\]
Let us define, for all $t\in [0,1]$ and $\phi\geq 0$,
\[H(t,\phi)=\beta\Hs(\partial^e\Om_t)+\int_{\partial^i\Om_t}\phi\ds-\int_{\Om_t}\phi^2\dx.\]

\begin{lemma}\label{bnnt11}
For any nonnegative $L^\infty$ function $\phi$, there exists some $t\in ]0,1[$ for which $H(t,\phi)\leq E_{\beta}(K,\Om)$.
\end{lemma}
\begin{proof}
Let $w=\phi-\frac{|\nabla u|}{u}$, then
\begin{align*}
H(t,\phi)-E_{\beta}(K,\Om)&=H(t,\phi)-H\left(t,\frac{|\nabla u|}{u}\right)\\
&=\int_{\partial^i\Om_t}w\ds+\int_{\Om_t}\left(\left(\frac{|\nabla u|}{u}\right)^2-\phi^2\right)\dx\\
&\leq \int_{\partial^i\Om_t}w\ds-2\int_{\Om_t}\frac{|\nabla u|}{u}w\dx\\
&=-\frac{1}{t}\frac{d}{dt}\left(t^{2}\int_{\Om_t}\frac{|\nabla u|}{u}w\dx\right),
\end{align*}
 where the last line is obtained by coarea formula on the level sets of $u$. So in particular
\[\int_{0}^{1}t(H(t,\phi)-E_{\beta}(K,\Om))dt\leq -\left[t^{2}\int_{\Om_t}\frac{|\nabla u|}{u}w\dx\right]_{t=0}^{t=1}=0,\]
which proves the lemma.
\end{proof}

\noindent {\it Proof of Theorem \ref{bnnt03}, (continuation).} Recall that $u^*$ is the solution on  $(B_1,B_R)$, which is radially symmetric and decreasing. Let $\phi$ be the dearrangement of $\frac{|\nabla u^*|}{u^*}$ on $\Om$ following  the level sets of $u$. To be more precise, for any $x\in\Om\setminus K$, let $r(x)>0$ be defined by the formula
\[|B_{r(x)}|=|\{u>u(x)\}|,\]
then we let $\phi(x)$ be the value of $\frac{|\nabla u^*|}{u^*}$ on $\partial B_{r(x)}$.\bigbreak

Let $t$ be chosen as in  Lemma \ref{bnnt11}, so that $H(t,\phi)\leq E_{\beta}(K,\Om)$. Assuming that $\frac{|\nabla u^*|}{u^*}\leq\beta$ on $\partial B_{r(x)}$, we have
\begin{align*}
H(t,\phi)&= \beta\Hs(\partial^e\Om_t)+\int_{\partial^i\Om_t}\phi\ds-\int_{\Om_t}\phi^2\dx\\
&\geq \int_{\partial \Om_t} \left(\frac{|\nabla u^*|}{u^*}\right)_{|\partial B_{r(x)}}\ds-\int_{\Om^*_t}\frac{|\nabla u^*|^2}{(u^*)^2}\dx&\\
&\geq \int_{\partial \Om_t^*} \frac{|\nabla u^*|}{u^*}\ds-\int_{\Om^*_t}\frac{|\nabla u^*|^2}{(u^*)^2}\dx&
\\
&=H^*\left(u^*_{|\partial B_{r(x)}},|\nabla u^*|/u^*\right)&\\
&=E_{\beta}(B_1,B_R)
\end{align*}
The inequalities above rely on the rearrangement properties, the isoperimetric inequality and on the hypothesis $\beta\geq \frac{|\nabla u^*|}{u^*}$.

We conclude by discussing when the assumption $\frac{|\nabla u^*|}{u^*}\leq \beta$ is verified. If $\beta\geq  n-1 $ the inequality $\frac{|\nabla u^*|}{u^*}\leq\beta$ is verified since $R \mapsto E_{\beta}(B_1,B_R)$ is decreasing on $[1, +\infty)$, by Lemma \ref{bnnt13}.\bigbreak
If instead $ n-2  <\beta< n-1 $, two situations occur. If $ M\ge |B_{R(n,\beta)}|$, then Lemma \ref{bnnt13} still works and the previous computation applies, with the same result. If $M\in [\om_n, |B_{R(n,\beta)}|]$ then we may consider $u^*$ the solution on $(B_1,B_{R(n,\beta)})$ restricted to $B_{R}$; it verifies $\frac{|\nabla u^*|}{u^*}\leq\beta$ and the same argument applies. We obtain that

$$E_\beta (K, \Om) \ge E_\beta (B_1, B_{R(n,\beta)}).$$
Since $ E_\beta (B_1, B_{R(n,\beta)})=  E_\beta (B_1, B_1)$, we conclude with the minimality of the couple $(B_1,B_1)$.

\bigbreak
The case $n \ge 3$, $\beta \leq  n-2 $ can not be treated directly with the $H$ function, but is a direct consequence of Corollary \ref{CorPenalized}. Indeed, it is enough to take  $\Lambda\to 0$ in Corollary \ref{CorPenalized}, since $ E_{\beta}(B_1,B_R)\geq E_{\beta}(B_1,B_1)$ for all $R\geq 1$ in this case.

\end{proof}
\begin{remark}\rm
Finally we have the following picture for problem \eqref{bnnt01} with $M= R^n \om_n$; the minimizer is $(B_1,B_r)$, where $r$ is defined through the following case disjunction.
\begin{itemize}[label=\textbullet]\setlength\itemsep{1em}
\item [(a)] If $n-1 \leq \beta $ then $r=R$. 
\item [(b)]  If $ n-2 < \beta< n-1 $, then defining $R(n,\beta)>\frac{n-1}{\beta}$ as the unique non-trivial solution of the equation $E_{\Theta}\left(B_1,B_{R(n,\beta)}\right)=n\beta \om_n$ we have
\begin{itemize}
\item [$ \bullet$] if $R(n,\beta)> R\geq 1$ then  $r=1$ 
\item [$ \bullet$] if $R=R(n,\beta) $ then  $r=1$ or $r=R$
\item  [$\bullet$] if $R> R(n,\beta)$ then $r=R$
\end{itemize}

\item [(c)]  If $\beta\leq  n-2 $ then $r=1$.
\end{itemize}
\end{remark}

\section{Preparatory results: existence and regularity of relaxed solutions}\label{bnnt07}

We analyze in this section the existence of a solution of problem \eqref{bnnt01} for general $\Theta$. Precisely, our purpose is to prove the following.
\begin{proposition}\label{ThExistence}
Let $M>\om_n$ and  $\Theta$ admissible, such that, for some dimensional constant $c_n$ (that will be specified)
\begin{equation}\label{HypT2}
M<\om_n+c_n\left(\inf_{0<s<1}\frac{\Theta(s/3)}{\Theta(s)}\right)^{2n}\int_{0}^{1}\frac{t^{2n-1}\mathrm{d}t}{\Theta(t)^{n}},
\end{equation}
Then problem \eqref{bnnt01} has a (sufficiently regular) solution $(K,\Om)$.
\end{proposition}
The word "sufficiently" above will be described later, and refers to the regularity that we need in the proofs of Theorems \ref{bnnt04} and \ref{bnnt03.2}. The proof of this proposition relies on  ideas inspired from the the relaxation strategy in the $SBV$ framework (defined below) as introduced in \cite{BG10} and from  \cite[Chapter 29]{M12} on the existence proof of minimal clusters for the perimeter. The main difficulty in our case comes from both the generality of the function $\Theta$ and from the fact that two measure constraints have to be satisfied simultaneously. \newline

As mentioned previously, the functional space used for our relaxation procedure is the $SBV$ space, on which more information may be found in the books \cite{B98,AFP00}. In few words, for an open set $\Om\subset\Rn$, $SBV(\Om)$ is defined as the set of functions $u\in L^1(\Om)$ such that $Du$ (the differential of $u$ in the sense of distribution) is a Radon measure that decomposes as
\[Du=\nabla u\mathscr{L}^n+(u_+-u_-)\nu_u\mathscr{H}^{n-1}\lfloor J_u,\]
where $\nabla u\in L^1(\Om)$ and $J_u$ is the set of jump points, meaning points $x\in \Rn$ such that $y\mapsto u(x+ry)$ converges in $L^1_\text{loc}(\Rn)$ as $r$ goes to $0$ to
\[u_+ 1_{y:y\cdot\nu_u >0}+u_- 1_{y:y\cdot\nu_u<0}\]
for some $u_+\neq u_-\in\R$, $\nu_u\in\mathbb{S}^{n-1}$. The set $J_u$ turns out to be rectifiable. We will make use of the following compactness result, which can be obtained by combination of \cite[Th. 2.3]{B98} and \cite[Th. 2.12]{B98} applied to $\phi(z,w)=\Theta(z)+\Theta(w)$.
\begin{proposition}\label{th_sbv_compactness}
Let $(u_i)_{i\in\mathbb{N}}$ be a sequence in $SBV_{\text{loc}}(\Rn)$ with values in $[0,1]$ such that for any $R>0$,
\[\sup_{i\in\mathbb{N}}\left(\int_{B_R}|\nabla u_i|^2+\Hs(J_{u_i}\cap B_R)\right)<\infty\]
then up to the extraction of a subsequence, there is a function $u\in SBV_\text{loc}(\R^n)$ such that
\begin{align*}
u_i&\underset{a.e.}{\longrightarrow}u\\
\nabla u_i&\underset{L^2_\text{loc}}{\rightharpoonup}\nabla u\\
\forall R>0,\ \int_{J_u\cap B_R}\left[\Theta(u_+)+\Theta(u_-)\right]d\Hs&\leq\liminf_{i\to\infty}\int_{J_{u_i}\cap B_R}\left[\Theta((u_i)_+)+\Theta((u_i)_-)\right]d\Hs
\end{align*}
\end{proposition}

Here is the main strategy.
\begin{itemize}[label=\textbullet]\setlength\itemsep{1em}
\item [\bf Step 1.] (Relaxation) We change the problem into finding a minimizer for the following problem

\begin{equation}\label{bnnt01.r}
\inf_{u \in SBV (\R^n), |\{u= 1\}| \ge \om_n, |\{u>0\} |\le M   }\E_{\Theta}(u),
\end{equation}
 where 
\begin{equation}\label{eq_relaxedenergy}
\E_{\Theta}(u)=\int_{\Rn}|\nabla u|^2\dx+\int_{J_u}(\Theta(u_+)+\Theta(u_-))\ds.
\end{equation}

This means that we will need afterward to prove that such a relaxed minimizer (meaning a minimizer of the relaxed problem) corresponds to a classical minimizer.
\item [\bf  Step 2.] (A priori regularity) We prove that we may restrict to considering functions that verify an a priori estimate of the form $u\geq \delta 1_{\{u>0\}}$ for some explicit $\delta$; for Robin boundary conditions this kind of estimate is a cornerstone of the regularity theory. This is the only place where we use the hypothesis \eqref{HypT2}. Similar estimates may be found in \cite{BL14}, \cite{BG15} and \cite{CK16} (see also \cite{BBC08} for a related problem). Moreover, we prove a concentration lemma that says all the support of $u$ - up to a set of measure $\eps$ - is contained in a union of $C\eps^{-n}$ unit cubes for some explicit constant $C$, and we prove a cut-off lemma that says that one may replace $u$ by $0$ outside a large enough cube and lower its energy by a controlled amount.
\item [\bf Step 3.] (Existence of a relaxed solution) We rely on Step 2 to prove the existence of a minimizer. This is done by the direct method,  considering a minimizing sequence, applying the concentration lemma a first time to translate the minimizing sequence such that it converges to a non-trivial solution, and then applying again the concentration lemma and cut-off lemma for an appropriately small $\eps$.
\item [\bf Step 4.]  (Regularity of the relaxed solution) Finally we prove the regularity of such a relaxed minimizer by using the theory of almost quasiminimizers of the Mumford-Shah functionnal to handle the support of $u$, and the theory of the regularity of Alt-Caffarelli problem inside $\Om$.
\end{itemize}

As mentioned above, we are first interested in the relaxed problem \eqref{bnnt01.r}. Notice that the generalized energy \eqref{eq_relaxedenergy} takes into account ``cracks'', meaning part of the jump set where $u$ may be nonzero on both sides. The general energy \eqref{eq_relaxedenergy} may also be used to define $E_\Theta(K,\Om)$ in a more general setting ; for any open set $\Om$, and any measurable $K\subset\Om$, we let
\[E_\Theta(K,\Om)=\inf\left\{\E_\Theta(u),\ u\in SBV(\Rn,[0,1]):|K\setminus \{u=1\}|=|\{u\neq 0\}\setminus\Om|=\Hs(J_u\setminus\partial\Om)=0\right\}\]
When $\Om$ is a smooth open set, this is coherent with the first definition.

\begin{lemma}\label{apriori}
Assume that \eqref{HypT2} is verified. There is a constant $\delta=\delta_{n,\Theta,M}>0$ such that for any admissible $u $ in \eqref{bnnt01.r}, there is some $t>\delta$ verifying $\E_{\Theta}(u1_{\{u>t\}})\leq \E_{\Theta}(u)$.
\end{lemma}

\begin{proof}
We actually show the stronger result; there is some $t>s>e^{-c_{n,\Theta}M}$ such that \[\E_{\Theta}(u1_{\{s<u<t\}^c})\leq \E_{\Theta}(u).\]

Let $\eps>0$, suppose $\E_{\Theta}(u1_{\{s<u<t\}^c})> \E_{\Theta}(u)$ for every $\eps<s<t<1$. We write
\begin{itemize}[label=\textbullet]\setlength\itemsep{1em}
\item $\Om\left(s,t\right)=\{s<u\leq t\}$.
\item $\gamma\left(s,t\right)=\int_{J_u}\left( 1_{s<u_+\leq t}+1_{s<u_-\leq t} \right)\ds$.
\item $h(t)=\Hs(\{u=t\}\setminus J_u)$.
\end{itemize}
Our hypothesis becomes that for every $\eta\in ]2\eps,\frac{2}{3}[$, $t\in ]0,\frac{\eta}{2}[$, $\E_{\Theta}(u1_{\{\eta-t<u<\eta+t\}^c})>\E_{\Theta}(u)$ so

\begin{equation}\int_{\Om\left(\eta-t,\eta+t\right)}|\nabla u|^2\dx+\Theta\left(\frac{1}{2}\eta\right)\gamma\left(\eta-t,\eta+t\right)\leq \Theta\left(\frac{3}{2}\eta\right) \left(h\left(\eta-t\right)+h\left(\eta+t\right)\right).\end{equation}

The proof is based on a lower bound of $\int_{\Om\left(\eta-t,\eta+t\right)}|\nabla u|\dx$ and an upper bound of $\int_{\Om\left(\eta-t,\eta+t\right)}|\nabla u|^2\dx$ that are in contradiction when $\eps$ is small enough.

\begin{itemize}[label=\textbullet]\setlength\itemsep{1em}
\item For all $t\in ]0,\frac{1}{2}\eta[$ we let:
\[f_\eta(t)=\int_{\eta-t}^{\eta+t}h(s) \mathrm{d}s=\int_{\Om\left(\eta-t,\eta+t\right)}|\nabla u|\dx.\]
$f_\eta$ is absolutely continuous and
\[f_\eta'(t)=h\left(\eta-t\right)+h\left(\eta+t\right).\]
Moreover,
\begin{align*}
\mbox{Then, } f_\eta(t)&\leq |\Om\left(\eta-t,\eta+t\right)|^{\frac{1}{2}}\left(\int_{\Om\left(\eta-t,\eta+t\right)}|\nabla u|^2\dx\right)^\frac{1}{2}\\
&\leq C_n\text{Per}(\Om\left(\eta-t,\eta+t\right))^\frac{n}{2(n-1)}\left(\Theta\left(\frac{3}{2}\eta\right)(h\left(\eta-t\right)+h\left(\eta+t\right))\right)^\frac{1}{2}\\
&\text{ by isoperimetric inequality}\\
&\leq C_{n}\left(h\left(\eta-t\right)+\gamma\left(\eta-t,\eta+t\right)+h\left(\eta+t\right)\right)^\frac{n}{2(n-1)}\left(\Theta\left(\frac{3}{2}\eta\right)(h\left(\eta-t\right)+h\left(\eta+t\right))\right)^\frac{1}{2}\\
&\leq C_{n}\Theta\left(\frac{3}{2}\eta\right)^{\frac{2n-1}{2n-2}}\Theta\left(\frac{1}{2}\eta\right)^{-\frac{n}{2n-2}} f_\eta'(t)^{\frac{2n-1}{2n-2}}.
\end{align*}
Raising to power $\frac{2n-2}{2n-1}$, after summation on $[0,t]$ 
\[\left(\int_{\Om\left(\frac{1}{2}\eta,\frac{3}{2}\eta\right)}|\nabla u|\dx\right)^{\frac{1}{2n-1}}\geq c_{n}t\Theta\left(\frac{3}{2}\eta\right)^{-1}\Theta\left(\frac{1}{2}\eta\right)^{\frac{n}{2n-1}}.\]
And so
\[\int_{\Om\left(\frac{1}{2}\eta,\frac{3}{2}\eta\right)}|\nabla u|\dx\geq c_{n}\eta^{2n-1} \Theta\left(\frac{3}{2}\eta\right)^{-(2n-1)}\Theta\left(\frac{1}{2}\eta\right)^{n}.\]
\item We let
\[g_\eta(t)=\int_{\Om\left(\eta-t,\eta+t\right)}|\nabla u|^2\dx,\ G_\eta(t)=\int_{0}^t g_\eta.\]
We begin by finding a upper bound for $G_\eta$. We take $t$ in $[0,\eta/2]$, then:
\begin{align*}
G_\eta(t)&\leq \int_{0}^{t}\Theta\left(\frac{3}{2}\eta\right) (h(\eta-s)+h(\eta+s))ds\\
&\leq \Theta\left(\frac{3}{2}\eta\right)|\Om\left(\eta-t,\eta+t\right)|^\frac{1}{2}G_\eta'(t)^\frac{1}{2}.
\end{align*}
Thus:
\[G_\eta'(t)G_\eta(t)^{-2}\geq \Theta\left(\frac{3}{2}\eta\right)^{-2}\left|\Om\left(\frac{1}{2}\eta,\frac{3}{2}\eta\right)\right|^{-1}.\]
We integrate from $t$ to $2t$ (up to supposing $t<\eta/4$) 
\[G_\eta(t)^{-1}\geq t\Theta\left(\frac{3}{2}\eta\right)^{-2}\left|\Om\left(\frac{1}{2}\eta,\frac{3}{2}\eta\right)\right|^{-1}.\]
Thus:
\[G_\eta(t)\leq  t^{-1}\Theta\left(\frac{3}{2}\eta\right)^2\left|\Om\left(\frac{1}{2}\eta,\frac{3}{2}\eta\right)\right|.\]
Since $g_\eta$ is increasing, then up to supposing $t<\eta/8$
\[g_\eta(t)\leq\frac{1}{t}\int_{t}^{2t}g_\eta\leq\frac{G_\eta(2t)}{t}\leq  t^{-2}\Theta\left(\frac{3}{2}\eta\right)^2\left|\Om\left(\frac{1}{2}\eta,\frac{3}{2}\eta\right)\right|.\]
\end{itemize}
Combining the previous inequalities, that are valid for  $t\in [0,\eta/8]$, we get
\begin{align*}
c_{n}\eta^{2n-1} \Theta\left(\frac{3}{2}\eta\right)^{-(2n-1)}\Theta\left(\frac{1}{2}\eta\right)^{n}&\leq \int_{\Om\left(\eta-t,\eta+t\right)}|\nabla u|\dx\\
&\leq |\Om\left(\eta-t,\eta+t\right)|^\frac{1}{2}\left(\int_{\Om\left(\eta-t,\eta+t\right)}|\nabla u|^2\dx\right)^\frac{1}{2}\\
&\leq t^{-1}\Theta\left(\frac{3}{2}\eta\right)\left|\Om\left(\frac{1}{2}\eta,\frac{3}{2}\eta\right)\right|.
\end{align*}
Taking $t=\eta/8$, we get:
\[\left|\Om\left(\frac{1}{2}\eta,\frac{3}{2}\eta\right)\right|\geq \frac{c_n\eta^{2n}}{\Theta\left(\frac{3}{2}\eta\right)^{2n}\Theta\left(\frac{1}{2}\eta\right)^{-n}},\]
for a constant $c_{n}$ that may be made explicit. Now, suppose that the quantity $\eps$ chosen at the start was of the form $3^{-K}$ for some $K\in\mathbb{N}^*$, and then by taking $\eta=\frac{2}{3^k}$ for $k=1,\hdots,K$, we get
\begin{align*}
M-\om_n&\geq |\{u>3^{-K}\}|\geq \sum_{k=1}^{K}\left|\Om\left(\frac{1}{3^{p}},\frac{1}{3^{p-1}}\right)\right|\\
&\geq c_n\sum_{k=1}^{K}\frac{3^{-2nk}}{\Theta\left(3^{-(k-1)}\right)^{2n}\Theta\left(3^{-k}\right)^{-n}}\\
&\geq c_n\left(\inf_{0<s<1}\frac{\Theta(s/3)}{\Theta(s)}\right)^{2n}\int_{3^{-K}}^{1}\frac{t^{2n-1}\mathrm{d}t}{\Theta(t)^{n}}.
\end{align*}
Thus if the hypothesis \eqref{HypT2} in the result is verified, this gives an upper bound on $K$.
\end{proof}

With this, we may replace any minimizing sequence $(u_i)$ by a minimising sequence $(u_i1_{\{u_i>t_i\}})$ such that \[\inf_{\{u_i>0\}}(u_i)\geq \delta\]
so we will now suppose that all the functions we consider verify this property.\\

For any $p\in\mathbb{Z}^n$, let us write $K_p=p+[0,1]^n$.
\begin{lemma}\label{nucl}
Let $u $ be admissible for \eqref{bnnt01.r} and  $\delta>0$ such that $u\geq \delta 1_{\{u>0\}}$. Let $\eps>0$. Then there exists a set $F=\cup_{i\in I}K_i$ where $I\Subset\mathbb{Z}^n$ is such that
\[|I|\leq C_{n}M\left(\frac{\Theta(\delta)^{-1}\E_{\Theta}(u)+M}{\eps}\right)^n,\ |\{u>0\}\setminus F|\leq \eps.\]
\end{lemma}
\begin{proof}
We first prove the following technical result.
Let $(K_i)_{i\in I}$ be a finite family of unit cubes, $F=\cup_{i\in I}K_i$, then
\[\max_{p\notin I}\left|\{u>0\}\cap K_p\right|\geq c_n \left(\frac{|\{u>0\}\setminus\cup_{p\notin I}K_p|}{|\{u>0\}|+\Theta(\delta)^{-1}\E_{\Theta}(u)}\right)^n.\]
For any $p\notin I$, we may write
\begin{align*}
|\{u>0\}\cap K_p|&= |\{u>0\}\cap K_p|^{\frac{1}{n}}|\{u>0\}\cap K_p|^{1-\frac{1}{n}}\\
&\leq C_n\left(\max_{q\notin I}|\{u>0\}\cap K_p|\right)^\frac{1}{n}\left(|\{u>0\}\cap K_p|+\Hs(J_u\cap K_p)\right)\\
&\text{ by the embedding }\mathrm{BV}(K_p)\hookrightarrow L^{\frac{n}{n-1}}(K_p)\\
&\leq C_n\left(\max_{q\notin I}|\{u>0\}\cap K_p|\right)^\frac{1}{n}\left(|\{u>0\}\cap K_p|+\Theta(\delta)^{-1}\E_{\Theta}(u|K_p)\right).\\
\end{align*}
 where the last term is defined as
\[\E_\Theta(u|\Om):=\int_{\Om}|\nabla u|^2\dx+\int_{J_u}\left(\Theta((u1_{\Om})_+)+\Theta((u1_{\Om})_-)\right)\ds\]
And so, by summing in $p\in \mathbb{Z}^n\setminus I$:
\[|\{u>0\}\setminus\cup_{q\notin I}K_q|\leq C_n\left(\max_{q\notin I}|\{u>0\}\cap K_p|\right)^\frac{1}{n}\left(|\{u>0\}|+\Theta(\delta)^{-1}\E_{\Theta}(u)\right),\]
which is the result.

We now construct $F$ by induction starting with $F_0=\emptyset$ and as long as $|\{u>0\}\setminus F_k|\geq \eps$, we take $I_{k+1}=I_k\cup\{p\}$ where $p\notin F_k$ is chosen with the previous lemma such that $\left|\{u>0\}\cap K_p\right|\geq c_n \left(\frac{\eps}{|\{u>0\}|+\Theta(\delta)^{-1}\E_{\Theta}(u)}\right)^n$. Suppose that this goes on until a rank $N$, then
\[M\geq |\{u>0\}|\geq \sum_{i\in I_N}|\{u>0\}\cap K_i|\geq c_n N \left(\frac{\eps}{|\{u>0\}|+\Theta(\delta)^{-1}\E_{\Theta}(u)}\right)^n,\]
so $N$ is bounded uniformly and conclude the proof of the lemma.
\end{proof}

For a closed set $F$, we shall write $d_F(x)=\inf_{y\in F}|x-y|$.

\begin{lemma}\label{trunc}
Let $u $ be admissible for \eqref{bnnt01.r} and  $\delta>0$ such that $u\geq \delta 1_{\{u>0\}}$. Then there exist constants $\tau_{n,\Theta,\delta}$, $C_{n,\Theta,\delta}$ such that for any closed set $F$ there exists some $r\in \left[0,C_{n,\Theta,\delta}|\{u>0\}\setminus F|^\frac{1}{n}\right]$ such that
\[\E_{\Theta}(u1_{\{d_F<r\}})\leq \E_{\Theta}(u)-\tau_{n,\Theta,\delta}|\{u>0\}\cap\{d_F>r\}|^{1-\frac{1}{n}}.\]
\end{lemma}

\begin{proof}
Let $\Om_r=\{u>0\}\cap \{d_F\geq r\}$, and $m(r)=|\Om_r|$. Suppose that the result we want to prove is not true for $r\in [0,r_1]$ for some $r_1>0$, meaning that for a constant $\tau>0$ that will be chosen later and for any $r\in ]0,r_1[$,
\[\E_{\Theta}(u|\Om_r)\leq \int_{\partial^*\Om_r\setminus J_u}\Theta(u)\ds+\tau |\Om_r|^{1-\frac{1}{n}},\]
where we remind $\E_\Theta(u|\Om):=\int_{\Om}|\nabla u|^2\dx+\int_{J_u}\left(\Theta((u1_{\Om})_+)+\Theta((u1_{\Om})_-)\right)\ds$.
Now, for any $r\in ]0,r_1[$,
\begin{align*}
m(r)^\frac{n-1}{n}&\leq C_n\Per(\Om_r) \ \text{(by isoperimetry)}\\
&\leq C_n\Theta(\delta)^{-1}\left(\int_{\partial^*\Om_r\setminus J_u}\Theta(u)\ds+\int_{\partial^*\Om_r\cap J_u}(\Theta((u1_{\Om_r})_-)+\Theta((u1_{\Om_r})_+))\ds\right)\\
&\leq C_n\Theta(\delta)^{-1}\left(\int_{\partial^*\Om_r\setminus J_u}\Theta(u)\ds+\E_{\Theta}(u|\Om_r)\right)\\
&\leq C_n\Theta(\delta)^{-1}\left(2\int_{\partial^*\Om_r\setminus J_u}\Theta(u)\ds+\tau |\Om_r|^{1-\frac{1}{n}}\right),
\end{align*}
so for $\tau:=\frac{\Theta(\delta)}{2C_n}$ this implies
\[m(r)^{\frac{n-1}{n}}\leq 4C_n\Theta(\delta)^{-1}\int_{\partial^*\Om_r\setminus J_u}\Theta(u)\ds\leq 4C_n\frac{\Theta(1)}{\Theta(\delta)}\Hs(\partial^*\Om_r\setminus J_u).\]
Moreover $\Hs(\partial^*\Om_r\setminus J_u)\leq-m'(r)$. So by integrating this for every $r\in ]0,r_1[$, we get
\[m(r_1)^{\frac{1}{n}}\leq m(0)^{\frac{1}{n}}-\frac{\Theta(\delta)}{4C_n\Theta(1)}r_1.\]
Since $m(r_1)\geq 0$ this means that necessarily
\[r_1\leq 4C_n \frac{\Theta(\delta)}{\Theta(1)}|\{u>0\}\setminus F|^\frac{1}{n}.\]
Thus, there is always some $r\in \left[0,8C_n \frac{\Theta(\delta)}{\Theta(1)}|\{u>0\}\setminus F|^\frac{1}{n}\right]$ such that
\[\E_{\Theta}(u1_{\{d_F<r\}})\leq \E_{\Theta}(u)-\tau|\{u>0\}\cap\{d_F>r\}|^{1-\frac{1}{n}},\]
where $\tau=\frac{\Theta(\delta)}{2C_n}$ as defined earlier.
\end{proof}

In all that follows we will need the following lemma for controlled infinitesimal volume exchange between measurable set, that may be found for instance in \cite[Lemma 29.13]{M12}.

\begin{lemma}\label{lem_volumeexchange}
Let $U$ be a connected open set, $E_1,\hdots,E_N$ be a measurable partition of $U$ such that, for every $i$, $|E_i\cap U|>0$. Then there are vector fields $X_{ij}$ with disjoint support such that for any $i,j,k$,
\[\int_{E_k}\dv(X_{ij})=\begin{cases}+1& \text{ if }k=i,\\-1& \text{ if }k=j,\\0& \text{ else}. \end{cases}\]
\end{lemma}

\begin{proof}
Consider the linear application
\[L:\begin{cases}\mathcal{C}^{\infty}_c(U)\to\R^N,\\ X\mapsto \left(\int_{E_i}\dv(X)\right)_{i=1,\hdots,N}.\end{cases}\]
The range of $L$ is included in $\{(m_1,\hdots,m_N):\sum_{i=1}^{N}m_i=0\}$. If it were not equal to this subspace, there would be a non trivial vector $(a_1,\hdots,a_N)$ independant of $(1,\hdots,1)$ such that, for any $X$, $\sum_{i=1}^{N}a_i\int_{E_i}\dv(X)=0$, meaning that $\nabla\left(\sum_{i=1}^{N}a_i1_{E_i}\right)=0$ in $\mathcal{D}'(U)$, which is a contradiction by the connectedness of $U$.
\end{proof}

\begin{proposition}\label{prop:existence}
 Under the hypotheses of Proposition \ref{ThExistence}, there exists a minimizer of $\E_\Theta$ in Problem \eqref{bnnt01.r}.
\end{proposition}
\begin{proof} 
Consider a minimizing sequence $(u_i)$; up to truncation we may suppose that $\inf_{\{u_i>0\}}(u_i)\geq \delta_{n,\Theta,m,M}$. We may apply Lemma \ref{nucl}  with $\eps=m/2$ to find sequences $(p_i^k)_{1\leq k\leq N,i\geq 0}$ such that $|\{u_i>0\}\setminus \cup_{k}K_{p_i^k}|\leq \frac{\om_n}{2}$. In particular, $|\{u_i=1\}\cap \cup_{k}K_{p_i^k}|\geq \frac{\om_n}{2}$ so for each $i$ there is some $k_i$ such that $|K_{p_i^{k_i}}\cap\{u= 1\}|\geq \frac{\om_n}{2N}$. Up to a translation we will suppose that $p_{i}^{k_i}=0$, such that $K_{p_i^{k_i}}=[0,1]^n$.\bigbreak

Now, by compactness arguments in $SBV$, see theorem \ref{th_sbv_compactness}, we know that up to an extraction $(u_i)$ converges almost everywhere (with lower semicontinuity on the energy) to a non-trivial limit $u\in SBV$, in particular such that $|\{u=1\}|>0$ and Fatou's lemma tells us that for any cube $K$, $|\{u=1\}\cap K|\geq \limsup_{i\to\infty}|\{u_i=1\}\cap K|$. We also denote $R:=(2M)^{1/n}$ to be such that $|[0,R]^n\cap\{u=0\}|\geq M$.

\begin{lemma}
Up to extraction, either $|\{u>0\}\cap[0,R]^n|>\limsup_{i\to\infty}|\{u_i>0\}\cap[0,R]^n|$, which we call the loose case, or
\[1_{\{u_i=1\}\cap [0,R]^n}\to 1_{\{u=1\}\cap[0,R]^n}\]
in the weak-$*$ sense, which we call the saturated case.
\end{lemma}
\begin{proof}
Suppose that $|\{u>0\}\cap[0,R]^n|=\limsup_{i\to \infty}|\{u_i>0\}\cap[0,R]^n|$. We denote $\nu$ a weak limit of the sequence of measures given by the density $1_{\{u_i=1\}}$; by hypothesis $\nu([0,R]^n)=|\{u=1\}\cap [0,R]^n|$. Moreover for any nonnegative continuous function $\varphi\in\mathcal{C}^{0}([0,R]^n,\R)$, by Fatou's lemma,
\[\nu(\varphi)=\lim_{i\to\infty}\int_{[0,R]^n}\varphi 1_{\{u_i=1\}}\dx\leq \int_{[0,R]^n}\varphi\limsup_{i\to\infty} 1_{\{u_i=1\}}\dx \leq \int_{[0,R]^n}\varphi1_{\{u=1\}}\dx,\]
which concludes the lemma.
\end{proof}

We now let $\eps>0$ to be a small number (it will be fixed later), and we find a union of $N(=N_{n,\Theta,m,M,\eps}$) cube (we also include cubes to cover $[0,R]^n$), denoted $F_i$, such that $|\{u_i>0\}\setminus F_i|\leq \eps$; by applying  Lemma \ref{trunc} to $(u_i,F_i)$, we find a radius $r_i\leq C_{n,\Theta,M}|\{u_i>0\}\setminus F_i|^{\frac{1}{n}}$ such that
\[\E_{\Theta}(u_i1_{\{d_{F_i}<r_i\}})\leq \E_{\Theta}(u_i)-\tau_{n,\Theta,M}|\{u_i>0\}\cap\{d_{F_i}>r_i\}|^{1-\frac{1}{n}}.\]
We let $v_i=u_i1_{\{d_F<r_i\}}$, and we now differentiate between the loose and the saturated case.
\begin{itemize}[label=\textbullet]\setlength\itemsep{1em}
\item Loose case: here we may choose $\eps<|\{u=1\}\cap [0,R]^n|-\limsup_{i\to\infty}|\{u_i=1\}\cap[0,R]^n|$. Since the support of $v_i$ is on a finite (not depending on $i$) number of cubes, they may be moved around so that $v_i$ is supported in a compact set. Then by $SBV$ compactness theorem \ref{th_sbv_compactness} we obtain that $v_i\to v$, such that  $\E_{\Theta}(v)\leq\liminf_{i\to\infty}\E_{\Theta}(v_i)=\inf_{|\{u=1\}|\geq\om_n,|\{u>0\}|\leq M}\E_\Theta(u)$. Moreover, $|\{v>0\}|=\lim_{i\to\infty}|\{v_i>0\}|\leq M$ and
\begin{align*}
|\{v=1\}|&\geq |\{u=1\}\cap[0,R]^n|-\limsup_{i\to\infty}|\{u_i=1\}\cap [0,R]^n|+\limsup_{i\to\infty}|\{v_i=1\}|\\
&\geq  |\{u=1\}\cap[0,R]^n|-\limsup_{i\to\infty}|\{u_i=1\}\cap [0,R]^n|-\eps +\om_n\\
&\geq \om_n,
\end{align*}
so $v$ is admissible and this proves the result.
\item Saturated case: based on the partition $\{u=0\},\{0<u<1\},\{u=1\}$ of $[0,R]^n$, where the first and last set have positive measure, there exists a vector field $\xi\in\mathcal{C}^\infty_c((0,R)^n,\Rn)$ such that
\[\int_{\{u=1\}}\dv(\xi)\dx=1,\ \int_{\{0<u<1\}}\dv(\xi)\dx=0,\ \int_{\{u=0\}}\dv(\xi)\dx=-1.\]
Moreover notice that $\phi_t(x):=x+t\xi(x)$ is a diffeomorphism with compact support for any small enough $t$; $\phi_t$ will be used to regulate the measure of $\{u_i=1\}$ after truncation. Using the weak convergence of the measure of the supports of the $(u_i)$, we may suppose that for any large enough $i$ and any small enough $t$ (not depending on $i$)
\begin{align*}
|\{u_i\circ\phi_t^{-1}=1\}|&=\int_{\{u_i=1\}}\det(D\phi_t)\dx=\int_{\{u_i=1\}}\left(1+t\dv(\xi)+t^2P_\xi(t)\right)\dx\\
&\geq |\{u_i=1\}|+\frac{t}{2},\\
|\{0<u_i\circ\phi_t^{-1}<1\}|&=\int_{\{0<u_i<1\}}\det(D\phi_t)\dx=\int_{\{0<u_i<1\}}\left(1+t\dv(\xi)+t^2P_\xi(t)\right)\dx\\
&\leq |\{0<u_i<1\}|+Ct^2.
\end{align*}
where $P_\xi$ is a polynomial of degree $n-2$ that depends on $\xi$.

Now, let $t_i=8(\om_n-\{v_i=1\})$. We know that $t_i= 8|\{u_i>0\}\setminus \{d_{F_i}>r_i\}|\leq 8\eps$; when $\eps$ is small enough,  for any $|t|<8\eps$ we know $\phi_t$ is a diffeomorphism and the estimates above hold for any large enough $i$. Consider
\[w_i(x)=v_i\circ \phi_{t_i}^{-1}\left( \left[\frac{\om_n}{\om_n-\frac{1}{4}t_i}\right]^{1/n}x\right).\]
Then
\begin{align*}
|\{w_i=1\}|&=\frac{\om_n-\frac{1}{4}t_i}{\om_n}|\{v_i\circ\phi_{t_i}^{-1}=1\}|\geq \frac{\om_n-\frac{1}{4}t_i}{\om_n}\left(|\{v_i=1\}|+\frac{1}{2}t_i\right),\\
&\geq |\{v_i=1\}|+\frac{t_i}{8}=\om_n\\
|\{0<w_i<1\}|&=\frac{\om_n-\frac{1}{4}t_i}{\om_n}|\{0<v_i\circ\phi_{t_i}^{-1}<1\}|\leq \frac{\om_n-\frac{1}{4}t_i}{\om_n}\left(|\{0<v_i<1\}|+Ct_i^2\right)\\
&\leq |\{0<v_i<1\}|\leq |\{0<u_i<1\}|,
\end{align*}
where in both lines we use that $t_i$ is taken arbitrarily small (less than $8\eps>0$, not depending on $i$). Thus $(w_i)$ is admissible. Then it may be computed with a similar method that
\begin{align*}
\E_{\Theta}(w_i)-\E_{\Theta}(v_i)&=\int_{\{\xi\neq 0\}}\left(\left|\left((I+t_iD\xi)^{*}\right)^{-1}\nabla u_i\right|^2\det(I_n+t_iD\xi)-|\nabla u|^2\right)\dx\\
&+\int_{J_u\cap\{\xi\neq 0\}}\left(\Theta(u_i^+)+\Theta(u_i^-)\right)\left(\nu_{J_{u_i}}^* (I_n+t_i D\xi)\nu_{J_{u_i}}-1\right)\ds\\
&\leq C_0t_i\text{ where }C_0\text{ does not depend on }i.
\end{align*}

And then, due to Lemma \ref{trunc},
\[\E_{\Theta}(w_i)\leq \E_{\Theta}(u_i)+8C_0|\{u_i>0\}\setminus \{d_{F_i}>r_i\}|-\tau_{n,\Theta,M}|\{u_i>0\}\cap\{d_{F_i}>r_i\}|^{1-\frac{1}{n}}\leq \E_{\Theta}(u_i),\]
where the last inequality is due to $|\{u_i>0\}\cap\{d_{F_i}>r_i\}|\leq \eps$ and $\eps$ is chosen small enough (depending on the flow, which was defined before $\eps$). So $(w_i)$ is an admissible minimizing sequence that is confined in a disjoint union of $N$ unit cubes, and up to moving these cubes we may suppose that $w_i$ has support in a certain ball $B$ not depending on $i$. So with the compactness result \ref{th_sbv_compactness}, it converges to a minimizer.
\end{itemize}

\end{proof}

\begin{lemma}
Let $u$ be a relaxed minimizer of \eqref{bnnt01.r}. Then $J_u$ is $\Hs$-essentially closed, and there exists a bounded open set $\Om$ with $\partial\Om=\overline{J_u}$ and a relatively closed set $K\subset\Om$ such that $(K,\Om)$ is a solution of \eqref{bnnt01}, associated to the function $u$, with $u\in H^1(\Om)\cap\mathcal{C}^{0,\frac{2}{n+2}}_\loc(\Om)$ and $K=\{u=1\}$.
\end{lemma}

\begin{proof}

 We assume without loss of generality that $|\{0<u<1\}|>0$; otherwise the minimizer is directly identified by the isoperimetric inequality.
\begin{itemize}[label=\textbullet]\setlength\itemsep{1em}
\item We first prove that $u$ is an almost quasiminimizer of the Mumford-Shah functionnal, meaning that there are constants $c_u,r_u>0$ such that for any ball $B_{x,r}$ with $r<r_u$ and any function $v\in SBV(\Rn)$ that differs from $u$ on $B_{x,r}$ only,
\begin{equation}\label{eqQuasimin}
\int_{B_{x,r}}|\nabla u|^2\dx+\Theta(\delta)\Hs(B_{x,r}\cap J_u)\leq \int_{B_{x,r}} |\nabla v|^2\dx+2\Theta(1)\Hs(B_{x,r}\cap J_v)+c_u|\{u\neq v\}|.
\end{equation}

Let $U$ be the union of two balls of arbitrarily small radius $\delta$ such that 
\begin{align}
 &|U\cap\{u= 1\}|,\ |U\cap \{0<u<1\}|,\ |U\cap\{u=0\}|>0\label{eq_condball1}\\
 &|U^c\cap\{u= 1\}|,\ |U^c\cap \{0<u<1\}|,\ |U^c\cap\{u=0\}|>0.\label{eq_condball2}
\end{align}
The second condition \eqref{eq_condball2} is automatic as soon as $\delta$ is small enough, and the first condition \eqref{eq_condball1} may be obtained by continuity of $x\mapsto |B_{x,\delta}\cap A|$ for $A=\{u=1\},\{0<u<1\},\{u=0\}$.
We may apply lemma \ref{lem_volumeexchange} for these three sets in $U$, thus obtaining three smooth vector fields $(X,Y,Z)$ with support in $U$ that transfer measure between $(\{u=0\},\{0<u<1\},\{u=1\})$, in particular
\begin{align*}
\int_{\{u=1\}}\dv(X)=-\int_{\{u=0\}}\dv(X)=1,\ &\int_{\{0<u<1\}}\dv(X)=0\\
\int_{\{u=1\}}\dv(Y)=-\int_{\{u=0\}}\dv(Y)=1,\ &\int_{\{0<u<1\}}\dv(Y)=0\\
\end{align*}
Write $\phi_t(x)=x+tX(x)$, $\psi_t(x)=x+tY(x)$, and $\Phi_{s,t}=\phi_s\circ\psi_t$. Consider the application
\[G:\begin{cases}\R^2\to \R^2\\ \left(s,t\right)\mapsto \left(\left|\Phi_{s,t}(\{u= 1\})\right|-\left|\{u= 1\}\right|,\left|\Phi_{s,t}(\{0<u<1\})\right|-\left|\{0<u<1\}\right|\right).\end{cases}\]
Then $G$ is smooth and $DG(0,0)=I_2$; there exists $\eps_0>0$ such that $G$ is invertible on $B_{\eps_0}$ with $\frac{1}{2}|\left(s,t\right)|\leq |G\left(s,t\right)|\leq 2|\left(s,t\right)|$ and $G(B_{\eps_0})\supset B_{\eps_0/2}$. Let $r>0$ and $x\in\Rn$ be such that $|B_{x,r}\cap U|=0$ and $|B_{x,r}|<\eps_0/4$, and $v\in SBV$ be such that $\{u\neq v\}\Subset B_{x,r}$.  Up to truncating $v$ from above (by $1$, that can only decrease the energy) and from below (by some $t\geq\delta$, by lemma \ref{apriori}) we assume that $\delta 1_{\{v>0\}}\leq v\leq 1$.\bigbreak

Let $(a,b)=(|\{u=1\}|-|\{v=1\}|,|\{0<u<1\}|-|\{0<v<1\}|)(\in B_{\eps_0/2})$, and $\left(s,t\right)=G_{|B_{\eps_0}}^{-1}(a,b)$. Then $v\circ \Phi_{s,t}^{-1}$ satisfies the measure constraints $(\om_n, M)$, so
\[\E_{\Theta}(u)\leq E_{\Theta}(v\circ \Phi_{s,t}^{-1}).\]
So
\[\int_{B_{x,r}}|\nabla u|^2\dx+\Theta(\delta)\Hs(B_{x,r}\cap J_u)\leq \int_{B_{x,r}} |\nabla v|^2\dx+2\Theta(1)\Hs(B_{x,r}\cap J_v)+R,\]
where
\begin{align*}
R&=\int_{U}\left(\left|(D\Phi_{s,t}^*)^{-1}\nabla u\right|^2\det(D\Phi_{s,t})-|\nabla u|^2\right)\dx\\
&+\int_{U\cap J_u}\left(\nu_u^\bot D\Phi_{s,t}\nu_u^\bot-1\right)\left(\Theta(u_-)+\Theta(u_+)\right)\ds\\
&\leq C(|s|+|t|)\leq 2C(|a|+|b|).
\end{align*}
 This proves that $u$ is an almost quasiminimizer at any positive distance of $U$. Now by our choice of $U$, the condition \eqref{eq_condball2} allows us to choose similarly a set $U'\subset \Rn\setminus\overline{U}$ that is a union of two balls of arbitrarily small radii such that
\[|U'\cap\{u= 1\}|,\ |U'\cap \{0<u<1\}|,\ |U'\cap\{u=0\}|>0\]
and by choosing similarly vector fields $X',Y',Z'$ with support in $U'$ that transit measure between these three sets, $u$ is an almost quasiminimizer at any positive distance of $U'$. Thus $u$ is an almost quasiminimizer, and this concludes the proof of the estimate \eqref{eqQuasimin}.
\item By \cite[Theorem 3.1]{BL14}, $u$ being an almost quasiminimizer of Mumford-Shah implies that $\overline{J_u}$ is essentially closed, meaning $\Hs(\overline{J_u}\setminus J_u)=0$. We now let $\Om$ be the union of the bounded connected components of $\Rn\setminus \overline{J_u}$, then $u\in H^1(\Om)$ and $\partial\Om=\overline{J_u}$.
\item Let us prove $\Om$ is bounded by proving an explicit lower density estimate; indeed, we let $r:=r_u$ as defined in the first point and considering $v=u1_{\Rn\setminus B_{x,\rho}}$ for any $x\in\Rn$ and $\rho\in ]0,r[$ in \eqref{eqQuasimin} we get
\[\int_{B_{x,\rho}}|\nabla u|^2\dx+\Theta(\delta)\Hs(B_{x,\rho}\cap J_u)\leq 2\Theta(1)\Hs(\partial B_{x,\rho}\cap \{u>0\}\setminus J_u)+c_u |B_{x,\rho}\cap \{u>0\}|.\]
Let $m(\rho)=|\Om\cap B_{x,\rho}|$, then by the isoperimetric inequality
\[c_n m(\rho)^{1-\frac{1}{n}}\leq \Hs(\partial\Om\cap B_\rho)+\Hs(\Om\cap \partial B_\rho)\leq \left(1+2\frac{\Theta(1)}{\Theta(\delta)}\right)m'(\rho)+c_u m(\rho).\]
So if $|B_{x,r/2}\cap\Om|>0$ then by integrating this estimate, $|B_{x,r}\cap \Om|\geq cr^n$ for a constant $c$ that does not depend on $x$. Since $|\Om|\leq M$ then there are at most $N\leq\frac{M}{cr^n}$ points $(x_i)_{i=1,\hdots,N}$ such that $|x_i-x_j|\geq r$ for all $i\neq j$ and $|\Om\cap B(x_i,r/2)|>0$, meaning that $\Om$ is bounded.
\item We prove that $u$ is locally Hölder, which implies the relative closedness of $K$ with $K=\{u=1\}$. 
Indeed let $K=\{u=1\}$ and $B\Subset\Om$ be a small ball inside $\Om$ such that $|B\cap K|$ and $|B\setminus K|$ are positive (such a ball exists as soon as $|\Om|>|K|$, otherwise $K$ and $\Om\setminus K$ would be disconnected, and $(K,K)$ would have strictly lower energy than $(K,\Om)$). Let $\xi\in\mathcal{C}_c^{\infty}(B,\R^n)$ be such that
\[\int_{B\cap K}\dv(\xi)\dx=1,\ \int_{B\setminus K}\dv(\xi)\dx=-1,\]
and let $\phi_t(x)=x+t\xi(x)$ be the associated diffeomorphism for a small enough $t$. Consider then any ball $B_{x,r}\Subset \Om$ such that
\[B_{x,r^{\frac{n}{n+2}}}\Subset \Om\setminus B,\]
we prove that provided $r$ is small enough (depending only on the choice of the flow and the parameters of the problem and not on $x$) there is a constant $C>0$ not depending on $x$ and $r$ such that
\[\int_{B_{x,r}}|\nabla u|^2\dx\leq Cr^{\frac{n^2}{n+2}}.\]
This directly implies the local Hölder continuity in $\Om\setminus B$ and the relative closedness of $K$ using classical integral growth argument (see for instance \cite[cor. 3.2]{HL11}), and the same may be done for another small ball $B'$ that has positive distance from $B$ which conclude the result. Let us now focus on this estimate.\bigbreak

Let $R=r^{\frac{n}{n+2}}$ and $h$ be the harmonic extension of $u_{|\partial B_{x,R}}$ on $B_{x,R}$ (that we extend simply by $u$ outside $B_{x,R}$). We suppose that $r$ is small enough such that $r<\frac{R}{2}$. Notice that $h$ might not be admissible, however $|\{h=1\}|\geq \om_n-|B_{R}|$; we let $t=2|B_{R}|$ and for a small enough $R$, we have
\[|\{h\circ \phi_t^{-1}=1\}|\geq \om_n,\]
so $h\circ\phi_{t}^{-1}$ is admissible. By comparison with the minimizer $u$ we get
\[\int_{B_{x,R}}|\nabla (u-h)|^2\dx\leq \int_{B}\left(\left|\left((I+tD\xi)^*\right)^{-1}\nabla u\right|^2\det(I+t D\xi)-|\nabla u|^2\right)\dx\leq CR^n.\]
for some constant $C$ that depends on $n$, $u_{|B}$ and on the flow $\xi$. Now, using the subharmonicity of $|\nabla h|^2$ on $B_{x,R}$,
\begin{align*}
\int_{B_{x,r}}|\nabla u|^2\dx&\leq 2\int_{B_{x,r}}|\nabla (u-h)|^2\dx+2\int_{B_{x,r}}|\nabla h|^2\dx\\
&\leq  CR^n+2\left(\frac{r}{R/2}\right)^n\int_{B_{x,R/2}}|\nabla h|^2\dx\\
&\leq CR^n+C'\frac{r^n}{R^2}\text{ by Cacciopoli inequality on the second term}.
\end{align*}
This ends the proof.

\end{itemize}

\end{proof}

\begin{remark}\label{bnnt22} 
\rm
Summarizing our results, we know that $\Om$ is an open set with rectifiable topological boundary such that $\Hs(\partial\Om)<\infty$ and  that the temperature $u\in H^1(\Om)$ is   $\mathcal{C}^{0,\frac{2}{n+2}}_\loc(\Om)$.  
\end{remark}
\section{The penalized problem: proof of Theorem  \ref{bnnt04}}\label{bnnt08}
In this section we prove Theorem  \ref{bnnt04}. 
For any $\Lambda>0$, we denote
\[E_{\Theta,\Lambda}(K,\Om)=E_{\Theta}(K,\Om)+\Lambda|\Om\setminus K|,\ \E_{\Theta,\Lambda}(u)=\E_{\Theta}(u)+\Lambda|\{0<u<1\}|.\]

We claim that it is enough to prove the result in the case for a function $\Theta$  which satisfies \eqref{HypTheta} and such that $\Theta(u)=\mathcal{O}_{u\to 0}(u^2)$. Indeed, for any small $\eps>0$ one may replace $\Theta$  on $[0,1]$ with a perturbation that verifies \eqref{HypTheta}, namely
\[\Theta^\eps(u)=\min\left((\Theta(1)+\eps)\left(\frac{u}{\eps}\right)^2,\Theta(u)+ \eps 1_{(0,1]}\right)\]
and define $E_{\Theta^\eps}$, $E_{\Theta^\eps,\Lambda}$, accordingly. Then  we know $E_{\Theta^\eps,\Lambda}$ is minimal on some $(B_1,B_{R^\eps})$, and  that the associated state function has the form
\[u^\eps(x)=\begin{cases}1 & \text{ if }|x|\leq 1,\\
1-\left(1- u^\eps(R^\eps)\right)\frac{\Phi_n(|x|)-\Phi_n(1)}{\Phi_n(R^\eps)-\Phi_n(1)} & \text{ if }1\leq |x|\leq R^\eps,\\
0&\text{ if }|x|> R^\eps,\end{cases}\]
 due to the general form of radial harmonic functions ($x\mapsto a+b\Phi_n(|x|)$). The associated energy takes the form
\begin{align*}
E_{\Theta^\eps,\Lambda}(B_1,B_{R^\eps})&=\left(1-u^\eps(R^\eps)\right)^2\frac{\Per(B_r)\Phi_n'(1)}{\Phi_n(R^\eps)-\Phi_n(1)}+\Per(B_{R^\eps})\Theta^\eps\left(u^\eps(R^\eps)\right).
\end{align*}

 $R^\eps$ is bounded uniformly in $\eps$; indeed, due to the penalization term,
\[|B_{R^\eps}|\leq \Lambda^{-1}E_{\Theta^\eps}(B_1,B_{R^\eps})\leq\Lambda^{-1} E_{\Theta^\eps}(B_1,B_1)=\Lambda^{-1}(\Theta(1)+\eps)\Per(B_1)\]
so we may suppose without loss of generality that $R^\eps\goto R$ and  $u^\eps(R^\eps)\goto l\in [0,1]$, and define
\[u(x)=\begin{cases}1 & \text{ if }|x|\leq 1,\\
1-\left(1- l\right)\frac{\Phi_n(|x|)-\Phi_n(1)}{\Phi_n(R)-\Phi_n(1)} & \text{ if }1<|x|< R,\\
0&\text{ if }|x|> R,\end{cases}\]
Then by lower semicontinuity of $\Theta$ and the fact that $\Theta(0)=0$ we get
\begin{align*}
E_{\Theta,\Lambda}(B_1,B_R)&\leq \E_{\Theta,\Lambda}(u)\leq \liminf_{\eps\to 0}E_{\Theta^\eps,\Lambda}(B_1,B_{R^\eps})
\end{align*}
and then for any admissible $v $ we get by dominated convergence
\[\E_{\Theta,\Lambda}(v)=\lim_{\eps\to 0}\E_{\Theta^\eps,\Lambda}(v)\geq \liminf_{\eps\to 0}E_{\Theta^\eps,\Lambda}(B_1,B_{R^\eps})\geq E_{\Theta,\Lambda}(B_1,B_R)\]

This is a similar method as what was used in \cite{BGN20} to handle a similar function with $\Theta(v)\sim 2\beta cv$ near $0$ for some constants $c,\beta>0$. 
\begin{lemma}
Problem \eqref{bnnt02} has a solution.
\end{lemma}
\begin{proof} 
We let $\Theta^\eps(u) $ defined above. Notice that 
$$\inf_{0<s<1}\frac{\Theta^\eps(s/3)}{\Theta^\eps(s)}\geq \min\left(\frac{1}{9},\inf_{0<s<1} \frac{ \Theta(s/3)+\eps}{\Theta(s)+\eps}\right)\ge \min\left(\frac{1}{9},\frac{ \eps}{\Theta(1)+\eps} \right)>0$$
 and $\Theta^\eps(v)=\mathcal{O}_{s\to 0}(s^2)$ so the hypothesis \eqref{HypT} is automatically verified. For what comes next we drop the $\eps$ to lighten the notations.\\
The results of Section \ref{bnnt07} apply; for any $M\geq \om_n$, there exists a minimizer $(K^M,\Om^M)$ of problem \eqref{bnnt01}. $M \mapsto E_\Theta(K^M,\Om^M)$ is nonincreasing, and we prove that it is lower semicontinuous. Indeed, consider $M_i\underset{i\to\infty}{\longrightarrow}M$, and consider the sequence of minimizers $(K^{M_i},\Om^{M_i})$, with their state function $(u_i)$. Then proceeding as in the proof of Proposition \ref{prop:existence}, we find a modified sequence $(w_i)$ obtained from $(u_i)$ through truncation with uniformly bounded support, such that its limit $w$ verifies $|\{w=1\}|\geq \om_n$, $|\{w>0\}|\leq M$, and $\E_\Theta(w)\leq \liminf_{i\to\infty}\E_\Theta(u_i)$; this prove the lower semi-continuity. Since

\[E_{\Theta,\Lambda}(K^M,\Om^M)\geq \Lambda(M-\om_n)\underset{M\to+\infty}{\longrightarrow}+\infty,\]
then there is a (non-necessarily unique) $M_\Lambda>0$ such that $[\om_n, +\infty) \ni M \mapsto E_{\Theta,\Lambda}(K^M, \Om^M)$ is minimal at $M=M_\Lambda$ and $(K^{M_\Lambda},\Om^{M_\Lambda})$ is a minimum of $E_{\Theta,\Lambda}$ with $|\Om|\le M$.
\end{proof}

We may now prove the main result of this section. We will say a set $A$ has the center property if any hyperplane that goes through the origin divides $A$ in two parts of same measure. This is in particular the case if $A$ has central symmetry. We will also call a minimizer $(K,\Om)$ ``centered'' if $K$ has the center property.
\begin{proof}[Proof of Theorem \ref{bnnt04}]
Consider $u$ a minimizer of $\E_{\Theta,\Lambda}$ with the associated sets $\Om=\{u>0\}$, $K=\{u=1\}$. We remind that we know $\Om$ to be open and $u\in H^1(\Om)$. When $e$ is a unit vector and $\lambda\in\mathbb{R}$, we will write
\[u^+_{e,\lambda}(x)=\begin{cases}
u(x)& \text{ if }x\cdot e\geq \lambda\\
u(S_{e,\lambda}(x))& \text{ if }x\cdot e< \lambda\\
\end{cases},\ 
u^-_{e,\lambda}(x)=\begin{cases}
u(S_{e,\lambda}(x))& \text{ if }x\cdot e\geq \lambda\\
u(x)& \text{ if }x\cdot e< \lambda\\
\end{cases},\]
where $S_{e,\lambda}$ is the reflexion relative to the hyperplane $\{x:x\cdot e=\lambda\}$.
\begin{itemize}[label=\textbullet]\setlength\itemsep{1em}
\item We first show that there exists a minimizer with the center property by building a minimizer with central symmetry. Indeed, consider $u$ a minimizer associated to $(K,\Om)$, $\lambda_1\in\mathbb{R}$ such that $\{x:x_1=\lambda_1\}$ cuts $K$ in half. Then $u_{e_1,\lambda_1}^+$ and $u_{e_1,\lambda_1}^-$ are also minimizers; we may in particular replace $u$ with $u_{e_1,\lambda_1}^+$ or $u_{e_1,\lambda_1}^-$ (the choice does not matter here) to suppose $u$ has a symmetry along $\{x:x_1=\lambda_1\}$. We do the same for $\{x:x_i=\lambda_i\}$ successively for $i=2,\hdots,n$; in the end we arrive to a minimizer
\[\tilde{u}=\left(\left.\hdots\left(u_{e_1,\lambda_1}^\pm\right)_{e_2,\lambda_2}^\pm\right)\hdots\right)_{e_n,\lambda_n}^\pm,\]
that is symmetric relative to every $S_{e_i,\lambda_i}$. Up to a translation of $\lambda$, $\tilde{u}$ is invariant for every $S_{e_i,0}$, and so it is invariant by their composition which is the central symmetry $x\mapsto -x$.

\item Suppose that $u$ is a minimizer (with $K=\{u=1\}$, $\Om=\{u>0\}$) where $K$ has the center property (we know such a minimizer exists from the previous point). We prove that the free boundary $\partial K\cap\Om$ is a union of spherical arcs centered at the origin.

If $\{0<u<1\}$ is empty then the problem is equivalent to the isoperimetric inequality and we are done, so we suppose that it is not. Consider $A$ a connected component of $\{0<u<1\}$, which is open by the regularity of minimizers.
The set $\partial A\cap\partial K$ is not empty, since it would otherwise mean that $\E_{\Theta}(u1_{\Rn\setminus A})<\E_{\Theta}(u)$.

Let $H_e=\{x:x\cdot e=0\}$ be a hyperplane going through the origin such that $H_e\cap A\neq\emptyset$, then $u_{e,0}^{+}$ is also a minimizer; in particular this means by analyticity of $u$ in $A$ that $\nabla u_{|H_e}$ is colinear to $H_e$. Since this is true for every $e$, then $u_{|A}$ is a radial function (restricted to a set $A$ that may not be radial), and since it is harmonic it has an expression of the form
\[u_{|A}(x)=a_A-b_A\Phi_n(|x|),\]
where $\Phi_n$ is the fundamental solution of the Laplacian (taken with the sign convention that it is increasing). Moreover, since $u$ is $1$ on $\partial K\cap\partial A$, we know this set is an arc of circle around $0$, and we call its radius $r_A$. While $r_A$ might depend on the component $A$, notice that the optimality condition for the Alt-Caffarelli problem near $\partial K\cap\Om$ is exactly that for a certain constant $\lambda>0$ that does not depend on $A$,
\[\forall x\in \partial K\cap\partial \{0<u<1\},\ |\nabla u(x)|=\lambda,\]
where $\nabla u(x)$ refers in this case to the limit of $\nabla u(y)$ as $y(\in A)\to x$. With $u(x)=1$ this exactly reduces to
\[\begin{cases}
a_A-b_A\Phi_n(r_A)&=1,\\
-b_A\Phi_n'(r_A)&=\lambda,
\end{cases}\]
so $(a_A,b_A)$ are fully determined by $r_A$.
\item We prove a reflexion lemma that will be useful for successive reflections.
\begin{lemma}
Let $u$ be a centered minimizer. Then for any vector $e$, $u_{e,0}^+$ is also a centered minimizer.
\end{lemma}
\begin{proof}
We let $v=u_{e,0}^+$. If $\{0<v<1\}$ has zero measure then as a solution of the isoperimetric problem we know $v$ is the indicator of a ball so we are done; we suppose it is not the case. We may consider $A$ a connected component of $\{0<v<1\}$ as previously and $Z\Subset \partial A\cap\partial K$ a small spherical cap that is open with respect to $\partial B_{r_A}$. Consider another vector $f$ that is completed in an orthogonal basis $f=f_1,f_2,\hdots,f_n$ of $\Rn$. As previously we may reflect $v$ successively as
\[\tilde{v}=\left(\left.\hdots\left(v_{f_1,\lambda_1}^\pm\right)_{f_2,\lambda_2}^\pm\right)\hdots\right)_{f_n,\lambda_n}^\pm,\]
Where the signs are chosen such that $v$ and $\tilde{v}$ coincide on a quadrant that has non-empty intersection with $Z$. By construction $\tilde{v}$ is centrally symmetric around the point $\lambda=(\lambda_1,\hdots,\lambda_n)$, so $Z$ is a spherical cap both around the origin and $\lambda$; this implies that $\lambda=0$ so in particular $\lambda_1=0$, which means that $H_{f,0}$ cuts $\{u_{e,0}^+=1\}$ in half; this is what we wanted to prove.\end{proof}

\item Consider a centered minimizer $u$ (it exists due to the first step, since there exists a centrally symmetric minimizer), such that $|\{0<u<1\}|>0$. We construct another centered minimizer $v$ such that for some connected component $A$ of $\{0<v<1\}$, $\partial A\cap\partial \{v=1\}$ contains a centered sphere. We do this by iteration of the following lemma.
\begin{lemma}\label{lemmait}
Let $u$ be a centered minimizer, let $A$ be a connected component of $\{0<u<1\}$ and let $\lambda\in ]1,2[$ and $D^{\partial B_r}_{x,\rho}\Subset \partial A\cap\partial K$ be a small spherical disk in the sphere of radius $r:=r_A$. Then there exists another centered minimizer $v$ obtained by a finite number of reflexions on $u$, associated to the sets $(K',\Om')=(\{v=1\},\{v>0\})$, such that for some connected component $A'$ of $\{0<v<1\}$, $D_{x,\lambda \rho}^{\partial B_r}\subset\partial K'\cap\partial A'$.
\end{lemma}
\begin{proof}
We will denote $D(x,\rho)$ the ball relative the the unit sphere. If $\rho>\frac{\pi}{2} r$ then one reflexion along the hyperplane $x^\bot$ is enough; we suppose that $\rho\leq\frac{\pi}{2} r$.\bigbreak

Let $\eps=1-\frac{\lambda}{2}$. For every $y\in \partial D(x,(1-\eps)\rho)$, the reflexion by the hyperplane going through $y$ and tangent to $\partial D(x,(1-\eps)\rho)$ reflects a neighbourhood of $[x,y]$ on a neighbourhood $N(y)$ of $[x',y]$ for some point $x'$ in the continuation of the geodesic $[xy)$ at distance $\lambda\rho$ from $x$. By compactness there is some finite set $(y_i)_{i\in I}$ in $\partial D(x,(1-\eps)\rho)$ such that $\cup_{i\in I}N(y_i)\supset D(x,\lambda\rho)$. By iterating the associated reflexions we obtain the result.
\end{proof}

\item There exists a centered minimizer $u$ such that $\{u=1\}$ is a ball. Indeed, consider a minimizer $(K,\Om)$ with a function $u$ as previously, such that there is some $r>0$ for which $\partial B_r\subset \partial K$. For every such $r$, the function $u_r=\max(u,1_{B_r})$ is also a minimizer associated to $(K\cup B_r, \Om\cup B_r)$ that has the same energy as $u$, thus $u=1$ on $B_r$. In particular this means that $u$ is radially decreasing on the whole component of $\Om$ that contains a small neighbourhood of $B_r$; let us denote $\Om_1$ this component and $\Om_2=\Om\setminus \Om_1$, as well as $u_i=u1_{\Om_i}$. Then $u=u_1+u_2$ where $u_1,u_2$ have disjoint support and $J_{u_1}\cup J_{u_2}\subset J_u$. If $u_2$ is not zero this means that $u$ is a disconnected minimizer; let us prove that this is not possible.\bigbreak

Indeed, let $m_i=|\{u_i=1\}|$ and $v_i(\cdot)=u_i\left(\Big (\frac{m_i}{\om_n}\Big)^{1/n}\cdot\right)$. Then $|\{v_i=1\}|=\om_n$ and
\[\E_{\Theta,\Lambda}(u_i)=\left(\frac{m_i}{\om_n}\right)^{1-\frac{2}{n}}\int_{\Rn}|\nabla v_i|^2\dx+\left(\frac{m_i}{\om_n}\right)^{1-\frac{1}{n}}\int_{J_{v_i}}\Theta(v_i)\ds+\frac{m_i}{\om_n}\Lambda|\{v_i>0\}|>\frac{m_i}{\om_n}\E_{\Theta,\Lambda}(v_i),\]
where the inequality is strict because $\E_{\Theta}(v_i)>0$ and $m_i<\om_n$ (or else one of the $v_i$ is zero, which we supposed is not the case). Now,
\[\E_{\Theta,\Lambda}(u)=\E_{\Theta,\Lambda}(u_1)+\E_{\Theta,\Lambda}(u_2)>\frac{m_1}{\om_n}\E_{\Theta,\Lambda}(v_1)+\frac{m_2}{\om_n}\E_{\Theta,\Lambda}(v_2)\geq \inf \E_{\Theta,\Lambda}.\]
This contradicts the fact that $u$ is a minimizer. Thus $u_2=0$ and $u=u_1$ with $\{u_1=1\}=B_r$, which turns out to be the ball of volume $\om_n$, this determines $r$ to be $1$ and in turns means that $u$ is of the form
\[u(x)=\varphi(|x|)1_\Om,\]
for some radial function $\varphi$ that takes the value $1$ on $[0,r]$ and some set $A$ that contains $B_1$, and takes the value $1-c(\Phi_n(|x|)-\Phi_n(1))$ on $\Om$ for some $c>0$.

\item Consider such a minimizer $u$ associated with the sets $(K,\Om)$, such that $\{u=1\}$ is the ball $B_1$, $\{u>0\}$ contains a neighbourhood of $B_1$ and $u(x)$ is a radial function $\varphi(|x|)$ restricted to a non-necessarily radial set $\Om$. Then we prove that for some $R>0$, $\varphi 1_{B_R}$ is also a minimizer.\bigbreak

 First replace $\Om$ with its spherical cap rearrangement $\Om^s$ around the axis $\{te_1,t>0\}$, meaning that for every $t>0$, $\Om^s\cap\partial B_{t}$ is a spherical cap centered in $te_1$ with the same $\mathscr{H}^{n-1}$-measure as $\Om\cap \partial B_t$. Said differently this means that $\Hs(\Om\cap \partial B_t)=\Hs(\Om^s\cap\partial B_t)$ for all $t>0$ with $\partial \Om^s=\{f(e\cdot e_1)e,e\in\mathbb{S}^{n-1}\}$ for some nondecreasing $f:[-1,1]\to (r,+\infty)$ that is continuous in $1$ by openness of $\Om^s$ (it may also be checked afterward that in our case, $f$ is continuous at $-1$ by minimality although we will not need it). A property of spherical cap rearrangement (see for instance \cite[Prop. 3 and Rem. 4]{MHH11}) is that
\[\int_{\partial^*\Om^s}\varphi(|x|)d\Hs\leq\int_{\partial^*\Om}\varphi(|x|)d\Hs\]
and $|\Om^s|=|\Om|$, $\int_{\Om^s}|\varphi'(|x|)|^2dx=\int_{\Om}|\varphi'(|x|)|^2dx$. As a consequence, $\varphi(|\cdot|)1_{\Om^s}$ is still a minimizer. Write $R:=f(1)$, for any small $\eps>0$, there is a small enough $\rho>0$ such that
\[f([1-\rho,1])\subset [R-\eps,R].\]


We may iterate Lemma \ref{lemmait} again to obtain a minimizer $(u_\eps,K_\eps,\Om_\eps)$ such that $u_\eps$ is still given by the same radial function $\varphi$ on a non-necessarily radial set, and such that
\[K_\eps=B_r,\ B_{R-\eps}\subset \Om_\eps\subset B_{R}.\]
Since this may be done for any arbitrarily small $\eps$, by lower semi-continuity as $\eps\to 0$ we obtain that $(\varphi(|\cdot|)1_{B_R},B_r,B_{R})$ is a minimizer.
\bigbreak

Notice we could have done the same for $(B_1,B_{F(-e_1)})$; if $F(-e_1)<F(e_1)$ it means we are in a case in which $\mathrm{argmin}\{R\ge 1 \mapsto E_{\Theta,\Lambda}(B_1,B_R)\}$ is not uniquely defined.
\end{itemize}
\end{proof}

\begin{corollary}\label{CorPenalized}
If, additionally to the hypotheses of Theorem \ref{bnnt04}, $\Theta$ is such that $ [1,+\infty[\ni R\to E_{\Theta}(B_1,B_R)$ is minimal at $1$, then for any $(K,\Om)$
\[E_{\Theta}(K,\Om)\geq E_{\Theta}(B_1,B_1).\]
\end{corollary}

\begin{proof}
$(B_1,B_1)$ is always the minimizer of $E_{\Theta,\Lambda}$ as $\Lambda\to 0$, hence the result.
\end{proof}
\begin{remark} \label{bnnt20}\rm
This covers the case $\beta\leq  n-2 $ for  $\Theta(u)=\beta u^2$.
\end{remark}

\section{Proof of Theorem \ref{bnnt03.2}}\label{bnnt09}

Part of the assertions of the Theorem  \ref{bnnt03.2}, namely the existence of a solution to problem \eqref{bnnt01} and its regularity, have already been proved in the preparatory section \ref{bnnt07} (see Proposition \ref{ThExistence} and Remark \ref{bnnt22}). 

\begin{proof}[Proof of Theorem \ref{bnnt03.2}] In order to complete the proof of Theorem  \ref{bnnt03.2}, we give the following.
\begin{lemma}
Let $(K,\Om)$ a solution of \eqref{bnnt01}. Then $K$ has locally finite perimeter in $\Om$ and $\partial K\cap\Om$ is analytic when $n=2$.
\end{lemma}
\begin{proof}
We refer to \cite{BL09} where a  closely related problem   is treated. The same arguments apply in our case. 
\end{proof}

\begin{lemma}
Let $\Theta$ be l.s.c. and nondecreasing. Let $(K,\Om)$ be a solution of \eqref{bnnt01} for $M=R^n\om_n$. Then either $|\Om|=M$ or $E_{\Theta}(K,\Om)=\inf_{r\in [1,R]}E_{\Theta}(B_1,B_r)$.
\end{lemma} 
\begin{proof}
Suppose there is a minimizer with measure strictly less than $M$.  Roughly speaking, in this case the measure constraint is not saturated so that the problem behaves under many aspects an unconstrained one. 

By compactness, there exists $(K,\Om)$ that is the minimizer with the lowest volume, associated to a function $u$. Indeed, consider $(K_i,\Om_i)$ a sequence of such minimizers, associated with functions $(u_i)$, then reproducing the existence proof in Section \ref{bnnt07} on the sequence $(u_i)$, we obtain a new minimizer with minimal volume,  still denoted $(K,\Om)$.\bigbreak

If $|\Om|=\om_n$ then we are done, so we suppose that this is not the case.\bigbreak

Let $\mathcal{H}$ be the set of every hyperplane that cuts $K$ in half. It is not quite identified with $\RP$ since there may be several parallel hyperplanes that cut $K$ in half if $K$ is not connected. It is, however, straightforward that $\mathcal{H}$ is a connected set for the natural topology given by $\{(H,v)\in\RP\times \Rn\text{ s.t. }v\in H^\bot\}$. For every hyperplane $H$, we write \[m(H)=\sup\left\{\left|H^+\cap\Om\right|,\left|H^{-}\cap\Om\right|\right\}\]
and we let
\begin{align*}
\mathcal{H}_{<}&=\left\{H\in\mathcal{H}: m(H)<\frac{M}{2}\right\},\\
\mathcal{H}_{=}&=\left\{H\in\mathcal{H}: m(H)=\frac{M}{2}\right\},\\
\mathcal{H}_{>}&=\left\{H\in\mathcal{H}: m(H)>\frac{M}{2}\right\}.\\
\end{align*}
Then $\mathcal{H}_{<}$ and $\mathcal{H}_{>}$ are open relatively to $\mathcal{H}$ by the continuity of $m$. The minimality of the volume of $\Om$ and the fact that $|\Om|<M$ implies that $\mathcal{H}_{=}$ is empty (otherwise we could construct a minimizer with stricly lower volume by reflection around an element of $\mathcal{H}_{=}$).\bigbreak

Finally, $\mathcal{H}_<$ is not empty, as may be seen by considering, for each $\theta\in\frac{\mathbb{R}}{\pi\mathbb{Z}}$, an hyperplane orthogonal to $\cos(\theta)e_1+\sin(\theta)e_2$ and using intermediate value theorem after making a full turn. By connectedness, this directly implies that $\mathcal{H}_{>}=\emptyset$ and so $\mathcal{H}_{<}=\mathcal{H}$. Now, let $H\in\mathcal{H}$, then $\Om$ may be reflected on both sides of $H$: since the measure of $\Om$ is minimal among minimizers, then
\[\left|H^-\cap\Om\right|=\left|H^+\cap\Om\right|.\]

Thus, for every hyperplane that cuts $K$ in half, $\Om$ is also cut in half; this means that we can use the same arguments than  for the penalized problem in Theorem \ref{bnnt04}: by successive reflections, we find a solution with minimal volume such that the free boundary of $u$ is $\partial B_1$. By spherical rearrangement around $\{te_1,t>0\}$ we obtain a new minimizer $(K,\Om)$. Now, if $\Om$ is not a ball, then $\left|H_{e_1}^+\cap\Om\right|>\left|H_{e_1}^-\cap\Om\right|$, which is a contradiction.
\end{proof}

\medskip
We prove now the last assertion of Theorem \ref{bnnt03.2}. We begin with an a priori estimate that is similar to Lemma \ref{apriori}, only that it is done with in mind the idea to find a lower bound arbitrarily close to $1$.
\begin{proposition}\label{proposition_high_cutoff}
Let $\Theta$ be l.s.c and nondecreasing. There exists a constant $C_n>0$ such that, if $\delta\in ]0,1[$ verifies
\[\delta +C_n\frac{\Theta(1)}{\sqrt{\Theta(\delta)}}(M-\om_n)^\frac{1}{2n}< 1,\]
then for all $u$ such that $|\{u\ge 1\}|\ge \om_n$, $|\{u> 0\}|\le M$, there is some $t>\delta$ such that $\E_{\Theta}(u1_{\{u>t\}})\leq \E_{\Theta}(u)$.
\end{proposition}

This lemma is applicable in particular when $\delta$ is close to $1$ and $M-\om_n$ is small.\\

\begin{proof} This proof is very similar to Lemma \ref{apriori}, so some steps are only briefly described.
We actually show the stronger result; there is some $t>s>\delta$ such that \[\E_{\Theta}\left(u1_{\{s<u<t\}^c}\right)\leq \E_{\Theta}(u).\]

Let $\eps=1-\delta>0$, suppose $\E_{\Theta}(u1_{\{s<u<t\}^c})> \E_{\Theta}(u)$ for every $1-\eps<s<t<1$. We write
\begin{itemize}[label=\textbullet]\setlength\itemsep{1em}
\item $\Om\left(s,t\right)=\{s<u\leq t\}$.
\item $\gamma\left(s,t\right)=\int_{J_u}\left( 1_{s<u_+\leq t}+1_{s<u_-\leq t}\right)\ds$.
\item $h(t)=\Hs(\{u=t\}\setminus J_u)$.
\end{itemize}
Let $\eta:=1-\frac{\eps}{2}$, our hypothesis becomes that for every, $t\in ]0,\frac{\eps}{2}[$, $\E_{\Theta}(u1_{\{\eta-t<u<\eta+t\}^c})>\E_{\Theta}(u)$ so

\begin{equation}\label{Comp}\int_{\Om\left(\eta-t,\eta+t\right)}|\nabla u|^2\dx+\Theta(1-\eps)\gamma(\eta-t,\eta+t)\leq \Theta(1) \left(h\left(\eta-t\right)+h\left(\eta+t\right)\right).\end{equation}

The proof is, as previously, based on a lower bound of $\int_{\Om\left(\eta-t,\eta+t\right)}|\nabla u|\dx$ and an upper bound of $\int_{\Om\left(\eta-t,\eta+t\right)}|\nabla u|^2\dx$ that are in contradiction when $\eps$ is too large.

\begin{itemize}[label=\textbullet]\setlength\itemsep{1em}
\item For all $t\in ]0,\frac{1}{2}\eps[$ we let:
\[f(t)=\int_{\eta-t}^{\eta+t}h=\int_{\Om\left(\eta-t,\eta+t\right)}|\nabla u|\dx.\]
$f$ is absolutely continuous and
\[f'(t)=h\left(\eta-t\right)+h\left(\eta+t\right).\]
Moreover,
\begin{align*}
f(t)&\leq |\Om\left(\eta-t,\eta+t\right)|^{\frac{1}{2}}\left(\int_{\Om\left(\eta-t,\eta+t\right)}|\nabla u|^2\dx\right)^\frac{1}{2}\\
&\leq C_n\text{Per}(\Om\left(\eta-t,\eta+t\right))^\frac{n}{2(n-1)}\left(\Theta(1)(h\left(\eta-t\right)+h\left(\eta+t\right))\right)^\frac{1}{2}\ \text{  by isoperimetric inequality}\\
&\leq C_{n}\Theta(1)^\frac{1}{2}\left(h\left(\eta-t\right)+\gamma\left(\eta-t,\eta+t\right)+h\left(\eta+t\right)\right)^\frac{n}{2(n-1)}\left(h\left(\eta-t\right)+h\left(\eta+t\right)\right)^\frac{1}{2}\\
&\leq C_{n}\left(\frac{\Theta(1)}{\Theta(1-\eps)}\right)^{\frac{n}{2(n-1)}}\Theta(1)^{\frac{1}{2}} f'(t)^{\frac{2n-1}{2n-2}},
\end{align*}
so:
\[f'(t)f(t)^{-\frac{2n-2}{2n-1}}\geq c_n\Theta(1)^{-\frac{n-1}{2n-1}}\left(\frac{\Theta(1)}{\Theta(1-\eps)}\right)^{-\frac{n}{2n-1}}.\]
We integrate on $[0,t]$:
\[\left(\int_{\Om\left(\eta-t,\eta+t\right)}|\nabla u|\dx\right)^{\frac{1}{2n-1}}\geq c_{n}\Theta(1)^{-\frac{n-1}{2n-1}}\left(\frac{\Theta(1)}{\Theta(1-\eps)}\right)^{-n}t.\]
And so:
\[\int_{\Om\left(\eta-t,\eta+t\right)}|\nabla u|\dx\geq c_{n}\Theta(1)^{-(n-1)}t^{2n-1}.\]
\item We let:
\[g(t)=\int_{\Om\left(\eta-t,\eta+t\right)}|\nabla u|^2\dx,\ G(t)=\int_{0}^t g\]
We begin by finding a upper bound for $G$. We take $t$ in $[0,\eps/2]$, then:
\begin{align*}
G(t)&\leq \int_{0}^{t}\Theta(1)( h(\eta-s)+h(\eta+s))\mathrm{d}s\\
&\leq c_{n}\Theta(1)|\Om\left(\eta-t,\eta+t\right)|^\frac{1}{2}G'(t)^\frac{1}{2}.
\end{align*}
Thus:
\[G'(t)G(t)^{-2}\geq c_{n}\Theta(1)^{-2}|\Om\left(1-\eps,1\right)|^{-1}.\]
We integrate from $t$ to $2t$ (up to supposing $t<\eps/4$):
\[G(t)^{-1}\geq c_{n}t\Theta(1)^{-2}|\Om\left(1-\eps,1\right)|^{-1}.\]
Thus:
\[G(t)\leq C_{n}\Theta(1)^2 t^{-1}|\Om\left(1-\eps,1\right)|.\]
Since $g$ is increasing, then up to supposing $t<\eps/8$:
\[g(t)\leq\frac{1}{t}\int_{t}^{2t}g\leq\frac{G(2t)}{t}\leq  c_{n}\Theta(1)^2t^{-2}|\Om\left(1-\eps,1\right)|.\]
\end{itemize}
Combining the previous inequalities, that are valid for  $t\in [0,\eps/8]$, we get:
\begin{align*}
c_{n}\Theta(1)^{-(n-1)}\left(\frac{\Theta(1)}{\Theta(1-\eps)}\right)^{-n}t^{2n-1}&\leq \int_{\Om\left(\eta-t,\eta+t\right)}|\nabla u|\dx\\
&\leq |\Om\left(\eta-t,\eta+t\right)|^\frac{1}{2}\left(\int_{\Om\left(\eta-t,\eta+t\right)}|\nabla u|^2\dx\right)^\frac{1}{2}\\
&\leq C_n\Theta(1) t^{-1}|\Om\left(1-\eps,1\right)|.
\end{align*}
Taking $t=\eps/8$, and with $|\Om(1-\eps,1)|\leq M-\om_n$, we get:
\[\eps\leq C_n\Theta(1)\Theta(1-\eps)^{-\frac{1}{2}}(M-\om_n)^\frac{1}{2n},\]
for a constant $C_{n}>0$ that only depends on $n$; this proves the result.
\end{proof}

\noindent {\it Proof of Theorem \ref{bnnt03.2} (continuation).}
Let $\Theta$ be l.s.c. and nondecreasing. Suppose moreover that it is of class $\mathcal{C}^1$ near $1$ and 
\[\frac{\Theta'(1)^2}{\Theta(1)}< 4(n-1).\]
Then there is some $M_0>\om_n$ depending on $n,\Theta$ such that, for any $M\in [\om_n,M_0]$ and for any admissible $u$ with $|\{u\ge 1\}|\ge \om_n$, $|\{u> 0\}|\le M$,
\[\E_\Theta(u)\geq E_{\Theta}(B_1,B_1)\left(=\Theta(1) n\om_n\right).\]

Denote for simplicity $\Om=\{u>0\}$, $K=\{u=1\}$, and let $\delta$ be defined as in the previous result (it is always possible when $M-\om_n$ is small enough), and suppose that it is close enough to $1$ such that $\Theta$ in $\mathcal{C}^{1}$ on $[\delta,1]$. Using the previous lemma, we lose no generality in supposing $u_{|\Om}\geq\delta$, and $|\Om|=M$. Then
\begin{align*}
\int_{\partial \Om}\Theta(u)\ds+\int_{\Om}\Theta'(u)|\nabla u|\dx&=\int_{\partial \Om}\Theta(u)\ds+\int_{\delta}^{1}\Theta'(t)\Per(\{u>t\};\Om)\mathrm{d}t\\
&= \int_{\partial \Om}\Theta(u)\ds-\int_{\delta}^{1}\Theta'(t)\Per(\{u>t\};\partial\Om)\mathrm{d}t\\
&\ +\int_{\delta}^{1}\Theta'(t)\Per(\{u>t\})\mathrm{d}t\\
&=\Theta(\delta)\Per(\Om)+\int_{\delta}^{1}\Theta'(t)\Per(\{u>t\})\mathrm{d}t\\
&\geq \Theta(\delta)\left(\frac{M}{\om_n}\right)^{1-\frac{1}{n}}\Per(B_1)+\int_{\delta}^{1}\Theta'(t)\Per(B_1)\mathrm{d}t\\
&=\Theta(1)\Per(B_1)+\Theta(\delta)\Per(B_1)\left(\left(\frac{M}{\om_n}\right)^{\frac{n-1}{n}}-1\right).
\end{align*}
Now, 
\[\int_{\Om}\Theta'(u)|\nabla u|\dx\leq \int_{\Om}\Vert\Theta'\Vert_{\infty,[\delta,1]}|\nabla u|\dx\leq \int_{\Om}|\nabla u|^2\dx+\frac{\Vert\Theta'\Vert_{\infty,[\delta,1]}^2}{4}(M-\om_n),\]
Thus,
\begin{align*}
\int_{\Om}|\nabla u|^2\dx+\int_{\partial\Om}\Theta(u)\ds-\Theta(1)\Per(B_1)&\geq \Theta(\delta)\Per(B_1)\left(\left(\frac{M}{\om_n}\right)^{\frac{n-1}{n}}-1\right)\\
&-\frac{\Vert\Theta'\Vert_{\infty,[\delta,1]}^2}{4}(M-\om_n).
\end{align*}
So we obtain the result as soon as
\[\frac{\Vert\Theta'\Vert_{\infty,[\delta,1]}^2}{\Theta(\delta)}\leq 4\Per(B_1)\frac{\left(\frac{M}{\om_n}\right)^{\frac{n-1}{n}}-1}{M-\om_n}.\]
As $M\to \om_n$, we may take $\delta\to 1$, and this gives the result.\bigbreak

\end{proof}

\begin{remark}\rm
Conversely, suppose that 
\[\frac{\Theta'(1)^2}{\Theta(1)}> 4(n-1),\]
then for every $R>1$ close enough to $1$, $E_{\Theta}(B_1,B_R)<E_{\Theta}(B_1,B_1)$.
 Indeed, let $\eps>0$ and \[u_\eps(x)=\begin{cases}1 & (B_1),\\ 1-\frac{|x|-r}{2}\Theta'(1) & (B_{1+\eps}\setminus B_1).\end{cases}\]
Then
\begin{align*}
E_{\Theta}(B_1,B_{1+\eps})&\leq \E_{\Theta}(u_\eps)=n\om_n(1+\eps)^{n-1}\Theta\left(1-\frac{1}{2}\Theta'(1)\eps\right)+\om_n\left((1+\eps)^n-r^n)\right)\frac{\Theta'(1)^2}{4}\\
&=E_{\Theta}(B_1,B_1)+\left( (n-1)\Theta(1)-\frac{1}{4}\Theta'(1)^2\right)\eps\Per(B_1)+o_{\eps\to 0}\left(\eps\right).
\end{align*}
The first-order term is negative and this proves the converse for a small enough $\eps$.
\end{remark}

\bibliographystyle{plain}
\bibliography{biblio}

\end{document}